\documentclass[9pt,a4paper]{amsart}

\usepackage{amssymb}
\usepackage[authoryear]{natbib}
\usepackage[final]{showkeys}
\usepackage{microtype}
\usepackage{color}
\usepackage{graphicx}
\usepackage[hmargin=3cm,vmargin={3.5cm,4cm}]{geometry}
\newtheorem{te}{Theorem}[section]
\newtheorem{example}{Example}
\newtheorem{definition}[te]{Definition}
\newtheorem{os}[te]{Remark}

\newtheorem{lem}[te]{Lemma}
\newtheorem{coro}[te]{Corollary}

\numberwithin{equation}{section}

\allowdisplaybreaks

\begin{document}

    \title{Fractional Klein--Gordon equations and related stochastic processes}

    \author{Roberto Garra$^1$}
    \address{${}^1$Dipartimento di Scienze di Base e Applicate per l'Ingegneria, ``Sapienza'' Universit\`a di Roma.}

    \author{Enzo Orsingher$^2$}
    \address{${}^2$Dipartimento di Scienze Statistiche, ``Sapienza'' Universit\`a di Roma.}

    \author{Federico Polito$^3$}
    \address{${}^3$Dipartimento di Matematica, Universit\`a degli Studi di Torino.}


    \date{\today}

    \begin{abstract}
        This paper presents finite-velocity random motions driven by fractional Klein--Gordon
        equations of order $\alpha \in (0,1]$.
        A key tool in the analysis is played by the McBride's theory which converts fractional
        hyper-Bessel operators into Erd\'elyi--Kober integral operators.

        Special attention is payed to the fractional telegraph process whose space-dependent
        distribution solves a non-homogeneous fractional Klein--Gordon equation.
        The distribution of the fractional telegraph process for $\alpha = 1$ coincides
        with that of the classical telegraph process and its driving equation
        converts into the homogeneous Klein--Gordon equation.

        Fractional planar random motions at finite velocity are also investigated, the corresponding
        distributions obtained as well as the explicit form of the governing equations.
        Fractionality is reflected into the underlying random motion because in each time interval
        a binomial number of deviations $B(n,\alpha)$ (with uniformly-distributed orientation)
        are considered. The parameter $n$ of $B(n,\alpha)$ is itself a random variable
        with fractional Poisson distribution, so that fractionality acts as a subsampling
        of the changes of direction.
        Finally the behaviour of each coordinate of the planar motion is examined
        and the corresponding densities obtained.

        Extensions to $N$-dimensional fractional random flights are envisaged
        as well as the fractional counterpart of the Euler--Poisson--Darboux
        equation to which our theory applies.
    \end{abstract}

    \maketitle

    \section{Introduction}

        In this paper we consider Klein--Gordon type fractional equations of
        the form
        \begin{equation}
            \label{i1}
            \left(\frac{\partial^2}{\partial t^2}-c^2\frac{\partial^2}{\partial x^2}\right)^{\alpha}u(x,t)=
            \lambda^2 u(x,t),\qquad \alpha \in (0,1].
        \end{equation}
        The equation \eqref{i1} in a natural way emerges within the
        framework of relativistic quantum mechanics for $\lambda^2<0$
        from the expression of relativistic energy \citep{sakurai}. Hereafter we simply call \eqref{i1}
        fractional Klein--Gordon equation,
        for any $\lambda \in \mathbb{R}$.
        By means of the transformation
        \begin{align*}
            w = \sqrt{c^2t^2-x^2},
        \end{align*}
        equation \eqref{i1} takes the form
        \begin{equation}
            \left(\frac{d^2}{dw^2}+\frac{1}{w}\frac{d}{dw}\right)^{\alpha}u(w)= \frac{\lambda^2}{c^{2\alpha}}u(w),
        \end{equation}
        where a fractional power of the Bessel operator
        \begin{equation}
            L_B = \frac{d^2}{dw^2}+\frac{1}{w}\frac{d}{dw}=\frac{1}{w^2}\left(w\frac{d}{dw}w\frac{d}{dw}\right)
        \end{equation}
        appears.
        The fractional Bessel operator $(L_B)^{\alpha}$ can be studied by
        means of the McBride approach to the fractional calculus \citep{mc,mc2,mc1}.
        In particular,
        \begin{align}
            (L_B)^{\alpha}f(w)&=\left(\frac{d^2}{dw^2}+\frac{1}{w}\frac{d}{dw}\right)^{\alpha}f(w)
            \nonumber =
            4^{\alpha}w^{-2\alpha}I_2^{0,-\alpha}I_2^{0,-\alpha}f(w),
        \end{align}
        where the integral operators $I_2^{0,-\alpha}$ are special cases
        of the Erd\'elyi--Kober fractional integrals \citep[see][formula (2.10)]{mc}
        \begin{equation}
            \label{mc1}
            I_m^{\eta, \alpha}f=
            \frac{x^{-m\eta-m\alpha}}{\Gamma(\alpha)}\int_0^x(x^m-u^m)^{\alpha-1}u^{m\eta}f(u) \,d(u^m),
        \end{equation}
        for $\alpha>0$ and $f$ belonging to a suitable functional space (see Section \ref{prel}).

        The telegraph equation (equation of damped vibrations of strings)
        \begin{equation}
            \label{i2}
            \frac{\partial^2 u}{\partial t^2}+2\lambda\frac{\partial u}{\partial t}=c^2\frac{\partial^2 u}{\partial x^2}
        \end{equation}
        can be reduced to the classical Klein--Gordon equation
        (\eqref{i1} with $\alpha =1$) by means of the transformation
        $u(x,t)=e^{-\lambda t}v(x,t)$. Equation \eqref{i2} governs the
        distribution of the telegraph process $\mathcal{T}(t)$, $t \ge0$,
        and the related Klein--Gordon equation directs the absolutely
        continuous component of the distribution of $\mathcal{T}(t)$
        (see, for example, \citet{ale}). The telegraph process is a
        finite-velocity one-dimensional random motion of which many
        probabilistic features are well known. In this paper we study fractional extensions of
        the telegraph process, denoted by $\mathcal{T}^{\alpha}(t)$, $t \ge 0$,
        whose changes of direction are somehow related to the fractional
        Poisson process $\mathcal{N}^{\alpha}(t)$, $t \ge 0$, with
        one-dimensional distribution given by \citep{Beghin}
        \begin{align}
            P\{\mathcal{N}^{\alpha}(t)=k\} = \frac{1}{E_{\alpha, 1}(\lambda t^{\alpha})}
            \frac{\left(\lambda t^{\alpha}\right)^k}{\Gamma(\alpha k +1)},
            \qquad \alpha \in (0,1], \: k= 0,1,\dots,
        \end{align}
        with
        \begin{equation}
            E_{\alpha, 1}(\lambda t^{\alpha})=\sum_{k=0}^{\infty}\frac{(\lambda t^{\alpha})^k}{\Gamma(\alpha
            k+1)}.
        \end{equation}
        The probability law of $\mathcal{T}^{\alpha}(t)$, $t\geq 0$, can be written as
        \begin{align}
            \label{disin}
            p^{\alpha}(x,t) = {} &\frac{1}{E_{\alpha, 1}(\lambda t^{\alpha})}\bigg[ct \sum_{k=1}^{\infty}
            \left(\frac{\lambda}{2^\alpha c^{\alpha}}\right)^{2k}
            \frac{(c^2t^2-x^2)^{\alpha k -1}}{\Gamma(\alpha k)\Gamma(\alpha k +1)}
            + \sum_{k=0}^{\infty}\left(\frac{\lambda}{2^\alpha c^{\alpha}}\right)^{2k+1}
            \frac{(c^2t^2-x^2)^{\alpha k +\frac{\alpha-1}{2}}}{[\Gamma(\alpha k +\frac{1+\alpha}{2})]^2}\bigg]\\
            \nonumber & +\frac{1}{2 E_{\alpha, 1}(\lambda
            t^{\alpha})}[\delta(x+ct)+\delta(x-ct)], \qquad \alpha \in (0,1], \; \mbox{for $|x|\leq ct$}.
        \end{align}
        In \eqref{disin}, multi-index Mittag--Leffler functions \citep{ky} of the form
        \begin{equation}
            \label{i3}
            E^{(2)}_{(\alpha_j), (\beta_j)}(x)=\sum_{k=0}^{\infty}\frac{x^k}{\prod_{j=1}^2\Gamma(\beta_j+k\alpha_j)}
        \end{equation}
        appear.
        We note that, for $\alpha_j = \beta_j = 1$,
        the multi-index Mittag--Leffler function \eqref{i3} coincides
        with the modified Bessel function of the second order.

        The conditional distributions $P\{\mathcal{T}^{\alpha}(t)\in dx|\mathcal{N}^{\alpha}(t)=n\}$, in analogy
        to the conditional laws of the telegraph process $\mathcal{T}(t)$ can be obtained by means of
        the order statistics \citep[see][]{ale}.
        The fractional telegraph-type process can thus be regarded as a continuous-time
        random motion with a rightward step (up to a Beta-distributed
        instant) and a leftward motion during the subsequent time span.

        We consider also the multidimensional Klein--Gordon-type fractional equation
        \begin{equation}
            \label{i4}
            \left(\frac{\partial^2}{\partial t^2}-c^2\Delta\right)^{\alpha}u(\mathbf{x},t)
            = \lambda^2 u(\mathbf{x},t), \qquad \mathbf{x}\in \mathbb{R}^N, \: t\geq 0,
            0<\alpha \le 1,
        \end{equation}
        and a particular attention is devoted to the planar case ($N=2$).
        For $N=2$ and $\alpha =1$, equation \eqref{i4} can be obtained
        by means of the exponential transformation
        \begin{align*}
            v(x,y,t) = e^{-\lambda t}u(x,y,t),
        \end{align*}
        from the planar telegraph equation (also called equation of planar vibrations
        with damping)
        \begin{equation}
            \frac{\partial^2 v}{\partial t^2}+2\lambda \frac{\partial v}{\partial t}
            = c^2\left(\frac{\partial^2}{\partial x^2}+\frac{\partial^2}{\partial y^2}\right)v,
        \end{equation}
        which governs the distribution of a planar random motion with infinite directions \citep[see][]{kol}.
        A time-fractional telegraph equation was examined in
        \citet{ptrf} and the related composition of the telegraph
        process with a reflecting Brownian motion analyzed for
        $\alpha =\frac{1}{2}$. Recently more general space-time
        fractional telegraph equations in $\mathbb{R}^{N}$ were
        investigated in \citet{Bruno} and their solutions derived as
        the distribution of a composition of stable processes at the
        inverse of linear combinations of stable subordinators.
        We are able to obtain a fractional planar random motion $(X^\alpha(t),Y^\alpha(t))$, $t \ge 0$,
        which generalizes
        that treated in \citet{kol} and has explicit distribution
        \begin{align}
            \label{i5}
            P\{X^{\alpha}(t)\in dx, Y^{\alpha}(t)\in dy\} & = \frac{dx \, dy\,\lambda}{2\pi c^{\alpha}
            E_{\alpha,1}(\lambda t^{\alpha})}
            \frac{E_{\alpha,\alpha}\left(
            \frac{\lambda}{c^{\alpha}} \left(\sqrt{c^2t^2-(x^2+y^2)}\right)^{\alpha}\right)}
            {\left(\sqrt{c^2t^2-(x^2+y^2)}\right)^{2-\alpha}} \\
            \nonumber & = \frac{dx \, dy\,\lambda\, u_{\alpha}(x,y,t)}{2\pi c^{\alpha}
            E_{\alpha,1}(\lambda t^{\alpha})}, \qquad (x,y) \in C_{ct},
        \end{align}
        where
        \begin{align*}
            C_{ct}=\{(x,y)\in \mathbb{R}^2: x^2+y^2\leq c^2t^2\}.
        \end{align*}
        The function $u_\alpha$ appearing in \eqref{i5} is a solution
        to the non-homogeneous fractional Klein--Gordon equation
        \begin{equation}
            \nonumber
            \left(\frac{\partial^2}{\partial t^2}-c^2\left\{\frac{\partial^2}{\partial x^2}
            +\frac{\partial^2}{\partial y^2}\right\}\right)^{\alpha}u_{\alpha}(x,y,t)=
            \lambda^2 u_{\alpha}(x,y,t)+\frac{\lambda c^{\alpha}}{\Gamma(-\alpha)}
            \left(\sqrt{c^2t^2-x^2-y^2}\right)^{-\alpha-2},
        \end{equation}
        which reduces to a homogeneous one in the classical case $\alpha = 1$.
        The random motions worked out in this paper develop at finite velocity and the support
        of their distributions is a compact set.
        For this reason the fractional models dealt with here substantially differ from those appeared
        so far in the literature \citep{ptrf}.
        The fractional planar random motion considered here can be described by a particle
        where the number of changes of direction
        coincides with a fraction $\alpha$ of the changes of
        direction of orientation of the classical model.

        The projection of the fractional planar motion $(X^{\alpha}(t), Y^{\alpha}(t))$ on the $x$-axis
        has probability density
        \begin{equation}
            p_{\alpha}(x,t)=\frac{1}{E_{\alpha,1}(\lambda
            t^{\alpha})}\frac{\lambda}{2^\alpha c^{\alpha}}
            \sum_{k=0}^{\infty}\left(\frac{\lambda}{2^\alpha c^{\alpha}}\right)^{k}
            \frac{\left(\sqrt{c^2t^2-x^2}\right)^{k\alpha-1}}{[\Gamma(\frac{\alpha k+1}{2})]^2}, \qquad x \in [-ct, +ct].
        \end{equation}
        Therefore the distribution $p_{\alpha}(x,t)$ does not possess singular components
        as its planar counterpart (as well as the one-dimensional fractional telegraph
        process).
        If $\alpha =1$, we retrieve the distribution (1.3) of
        \citet{ale2}
        which is expressed in terms of Struve functions.

        A solution of the $N$-dimensional fractional Klein--Gordon equation has been obtained in the form
        \begin{equation}
            \label{i6}
            u_\alpha(\mathbf{x},t)= \sum_{k=0}^{\infty}\left(\frac{\lambda}{2^\alpha c^{\alpha}}\right)^{2k}
            \frac{\left(\sqrt{c^2t^2-\sum_{j=1}^N x_j^2}\right)^{2\alpha k+2\alpha-2}}{\Gamma(\alpha
            k+\alpha+\frac{N-1}{2})\Gamma(\alpha k+\alpha)}, \qquad
            \sum_{j=1}^N x_j^2\leq c^2t^2.
        \end{equation}
        From \eqref{i6}, the following conditional distribution can be extracted.
        \begin{equation}
            \label{i7}
            P\{X_1(t)\in dx_1, \dots, X_N(t)\in dx_N|\mathcal{N}^{\alpha}(t)=k\}
            =\frac{\Gamma(\frac{k\alpha+N}{2})\left(\sqrt{c^2t^2-\|x_n\|^2}\right)^{\alpha k-2}}
            {(ct)^{\alpha k +N-2}\Gamma(\frac{\alpha k}{2})\pi^{N/2}}.
        \end{equation}
        If $N=2$ the density \eqref{i7} coincides with the conditional distribution of $(X^{\alpha}(t), Y^{\alpha}(t))$,
        and for $\alpha =1$, we recover result (5)
        of \citet{kol}. For $N=4$, $\alpha=2$ we extract from \eqref{i7} the distribution (3.2) of \citet{ale2}.
        An extension of this
        theory for $\alpha >1$ is considered below.

        The last section of this paper is devoted to higher-order fractional Bessel-type equations of the form
        \begin{equation}
            \Biggl(\frac{1}{w^{n}}\Biggl(\underbrace{w\frac{d}{dw}w\frac{d}{dw}\dots w\frac{d}{dw}}_{\text{$n$
            derivatives}}\Biggr)\Biggr)^{\alpha}f(w)=\lambda^n f(w).
        \end{equation}
        These higher-order Bessel equations arise within the framework of cyclic motions in $\mathbb{R}^N$ with the
        minimal number $N+1$ of velocities directed on the edges of a hyperpolyhedron
        \citep[see for example][]{lac}. A special attention is devoted to the case of three orthogonal
        directions, where the distribution of
        $(X(t),Y(t))$, for $\alpha =1 $ can be expressed in terms of third-order Bessel functions
        \begin{align*}
            I_{0,3}(x)=\sum_{k=0}^{\infty}\left(\frac{x}{3}\right)^{3k}\frac{1}{(k!)^3}.
        \end{align*}
        An application of the McBride theory of fractional powers of differential operators to the Euler--Poisson--Darboux
        fractional equations is also considered.

    \section{Preliminaries on fractional hyper-Bessel operators}

        \label{prel}
        Our starting point is the generalized hyper-Bessel operator, considered in \citet{mc},
        \begin{equation}
            \label{L}
            L=x^{a_1}Dx^{a_2}\dots x^{a_n}Dx^{a_{n+1}},
        \end{equation}
        where $n$ is an integer number, $a_1,\dots, a_{n+1}$ are complex numbers and $D=d/dx$.
        Hereafter we assume that the coefficients $a_j$, $j=1,\dots, n+1$ are real numbers.
        The operator $L$ generalizes the classical $n$-th order hyper-Bessel operator
        \begin{align*}
            L_{B_n}=x^{-n}\underbrace{x\frac{d}{dx}x\frac{d}{dx}\dots x\frac{d}{dx}}_{\text{$n$ times}}.
        \end{align*}
        The operator $L$ defined in \eqref{L} acts on the functional space
        \begin{equation}
            F_{p,\mu}=\{f \colon x^{-\mu}f(x)\in F_p\},
        \end{equation}
        where
        \begin{equation}
            F_p=\{f\in C^{\infty} \colon x^k \frac{d^k f}{dx^k} \in L^p, k=0,1,\dots\},
        \end{equation}
        for $1 \leq p < \infty$ and for any complex number $\mu$ \citep[see][for details]{mc2, mc1}.
        The following lemma gives an alternative representation of the operator $L$.
        \begin{lem}
            \label{duepuntouno}
            The operator $L$ in \eqref{L} can be written as
            \begin{equation}
                \label{Lo}
                L f= m^{n}x^{a-n}\prod_{k=1}^n x^{m-m b_k}D_m x^{m b_k} f,
            \end{equation}
            where
            \begin{align*}
                D_m := \frac{d}{d x^m}=m^{-1}x^{1-m}\frac{d}{dx}.
            \end{align*}
            The constants appearing in \eqref{Lo} are defined as
            \begin{align*}
                a=\sum_{k=1}^{n+1}a_k, \qquad m= |a-n|,
                \qquad b_k= \frac{1}{m}\left(\sum_{i=k+1}^{n+1}a_i+k-n\right), \quad k=1, \dots, n.
            \end{align*}
        \end{lem}
        For the proof, see lemma 3.1, page 525 of \citet{mc}.

        \begin{example}
            Let us consider as a first example, the operator
            \begin{align*}
                L=\frac{d^2}{dx^2}+\frac{1}{x}\frac{d}{dx}=\frac{1}{x^2}\left( x\frac{d}{dx}x\frac{d}{dx} \right),
            \end{align*}
            that is a special case of \eqref{L} with $a_1=-1$, $a_2= 1$, $a_3=0$, $n=m= 2$, $a=0$,
            $b_1=b_2=0$.
            By Lemma \ref{duepuntouno}, we have that
            \begin{align}
                L&=\frac{4}{x^2}\prod_{k=1}^2x^{2-2b_k}D_2x^{2b_k}
                \nonumber = \frac{4}{x^2}\left(x^2D_2\right)\left(x^2D_2\right)\\
                \nonumber& = \frac{4}{x^2}\left(\frac{x}{2}\frac{d}{dx}\right)
                \left(\frac{x}{2}\frac{d}{dx}\right)= \frac{1}{x}\frac{d}{dx}+\frac{d^2}{dx^2}.
            \end{align}
        \end{example}
        In the analysis of the integer power (as well as the fractional power) of the operator $L$,
        a key role is played by $D_m$
        appearing in \eqref{Lo}.

        \begin{lem}
            Let $r$ be a positive integer, $a<n$, $f\in F_{p,\mu}$ and
            \begin{align*}
                b_k\in A_{p,\mu,m}:=\{\eta \in \mathbb{C}
                \colon \Re(m\eta+\mu)+m\neq \frac{1}{p}-ml, \: l=0, 1, 2,\dots\}, \qquad k=1,\dots, n.
            \end{align*}
            Then
            \begin{equation}
                L^r f= m^{nr}x^{-mr}\prod_{k=1}^n I^{b_k,-r}_m f,
            \end{equation}
            where, for $\alpha >0$ and $\Re(m\eta+\mu)+m > \frac{1}{p}$
            \begin{equation}
                \label{mc1-2}
                I_m^{\eta, \alpha}f=
                \frac{x^{-m\eta-m\alpha}}{\Gamma(\alpha)}\int_0^x(x^m-u^m)^{\alpha-1}u^{m\eta}f(u)\, d(u^m),
            \end{equation}
            and for $\alpha\leq 0$
            \begin{equation}
                \label{mc2}
                I_m^{\eta, \alpha}f=(\eta+\alpha+1)I_m^{\eta, \alpha+1}f+\frac{1}{m} I_m^{\eta, \alpha+1}
                \left(x\frac{d}{dx}f\right).
            \end{equation}
        \end{lem}
        For the proof, consult \citet{mc}, page 525.
        Then, it is possible to give a fractional generalization $L^{\alpha}$ of the operator $L$ with the
        following definition \citep[for further details see][page 527]{mc}.

        \begin{definition}
            Let $m= n-a>0$, $\alpha$ any complex number, $b_k\in A_{p,\mu,m}$, for $k=1,\dots, n$. Then, for
            any $f(x)\in F_{p,\mu}$
            \begin{equation}
                \label{pot}
                L^{\alpha}f=m^{n\alpha}x^{-m\alpha}\prod_{k=1}^{n}I_{m}^{b_k,-\alpha}f.
            \end{equation}
        \end{definition}
        In this paper we will consider however only $\alpha \in \mathbb{R}^+$.

        The relation between the two lemmas above emerges directly from the analysis of the mathematical
        connection between the power of the operator $D_m$ and the generalized fractional
        integrals $I_m^{\eta, \alpha}$, as we are going to discuss.
        In order to understand this relationship we introduce the following operator \citep{mc2}
        \begin{equation}
            \label{kob}
            I_m^{\alpha}f= \frac{m}{\Gamma(\alpha)}\int_0^x(x^m-u^m)^{\alpha-1}u^{m-1}f(u) \, du, \qquad \alpha>0,
        \end{equation}
        which is connected to \eqref{mc1-2} by means of the simple relation
        \begin{align*}
            I_m^{\alpha}f=x^{m\alpha}I_m^{0,\alpha}f,
        \end{align*}
        which is valid for all $\alpha \in \mathbb{R}$.

        It is quite simple to prove that
        \begin{align}
            I_m^{\alpha}f=(D_m)I_m^{\alpha+1}f
            \nonumber & =\frac{m}{\Gamma(\alpha+1)}D_m\int_0^x(x^m-u^m)^{\alpha}u^{m-1}f(u) \, du
            = \underbrace{D_m\dots D_m}_{\text{$r$ times}}I_m^{\alpha+r}f.
        \end{align}
        If $\alpha =-r$, we have that
        \begin{align*}
            I_m^{-r}f=\underbrace{D_m\dots D_m}_{\text{$r$ times}}I_m^{0}f=(D_m)^r f.
        \end{align*}
        For a real number $\alpha$, the same relationship is extended in the form
        \begin{equation}
            I_m^{-\alpha}f=\left(D_m\right)^{\alpha}f.
        \end{equation}
        Since the semigroup property holds for the Erd\'elyi--Kober operator \eqref{kob}, we have that
        \begin{align}
            (D_m)^{\alpha}f & = (D_m)^{n}(D_m)^{\alpha-n}f= I_m^{-n}I_{m}^{n-\alpha}f \\
            & =\frac{m}{\Gamma(n-\alpha)}(D_m)^n
            \int_0^x(x^m-u^m)^{n-\alpha-1}u^{m-1}f(u) \, du. \notag
        \end{align}
        Finally we observe that, for $m=1$ we recover the definition of Riemann--Liouville fractional
        derivative.

    \section{Fractional Klein--Gordon equation}

        \label{piccolo-gordon}
        Let us consider the following fractional Klein--Gordon equation
        \begin{equation}
            \label{KG}
            \left(\frac{\partial^2}{\partial t^2}-c^2\frac{\partial^2}{\partial x^2}
            \right)^{\alpha}u_{\alpha}(x,t)=\lambda^2 u_{\alpha}(x,t),\qquad x\in \mathbb{R}, \:
            t\geq 0, \: \alpha \in (0,1].
        \end{equation}
        The classical Klein--Gordon equation ($\alpha=1$, $\lambda^2<0$)
        emerges from the quantum relativistic energy equation \citep{sakurai}
        \begin{equation}
            E^2=p^2c^2+m^2c^4,
        \end{equation}
        and inserting the quantum mechanical operators for energy and
        momentum, i.e.\ $E=i\hbar\frac{\partial}{\partial t}$ and
        $p= -i\hbar\frac{\partial}{\partial x}$, where $c$ is
        the light velocity and $\hbar$ the Planck constant. In this
        framework the constant $\lambda^2$ appearing in \eqref{KG}
        reads $\lambda^2=-m^2c^4/\hbar^2$.

        The equation \eqref{KG}, for $\alpha=1$, appears also in the
        context of Maxwell equations, of damped vibrations of strings
        and in the treatment of the telegraph processes in probability.
        The fractional Klein--Gordon equation was recently studied in the
        context of nonlocal quantum field theory, within the stochastic
        quantization approach \citep[see][and the references therein]{li}.
        The fractional power of D'Alembert operator has been considered
        by \citet{Bol} and \citet{Lamb}, with different approaches.

        The transformation
        \begin{align}
            \begin{cases}
                \nonumber z_1 = ct+x,\\
                \nonumber z_2 = ct-x,
            \end{cases}
        \end{align}
        reduces \eqref{KG} to the form
        \begin{equation}
            \label{k1}
            \left(4c^2\frac{\partial}{\partial z_1}\frac{\partial}{\partial z_2}
            \right)^{\alpha}u_{\alpha}(z_1,z_2)= \lambda^2 u_{\alpha}(z_1,z_2).
        \end{equation}
        The partial differential equation \eqref{k1} involves
        in fact Riemann--Liouville fractional derivatives with respect to the
        variables $z_1$ and $z_2$ \citep[see][Section 2.3]{pod}.
        The further transformation
        $w= \sqrt{z_1\,z_2}$
        gives the fractional Bessel equation
        \begin{equation}
            \label{k2}
            \left(\frac{d^2}{dw^2}+\frac{1}{w}\frac{d}{dw}\right)^{\alpha}u_\alpha(w)=(L_B)^{\alpha}u_\alpha(w)
            = \frac{\lambda^2}{c^{2\alpha}}u_\alpha(w).
        \end{equation}
        The Bessel operator
        \begin{align*}
            L_{B}=\frac{d^2}{dw^2}+\frac{1}{w}\frac{d}{dw}
        \end{align*}
        appearing in
        \eqref{k2} is a special case of $L$, when $n=2$, $a_1=-1$, $a_2= 1$,
        $a_3=0$. By definition \eqref{pot} and Lemma \ref{duepuntouno} we have that $m= 2$,
        $b_1=b_2=0$ and thus
        \begin{equation}
            \label{k3}
            (L_B)^{\alpha}f(w)=4^{\alpha}w^{-2\alpha}I_2^{0,-\alpha}I_2^{0,-\alpha}f(w).
        \end{equation}
        For us the following lemma plays a relevant role.
        \begin{lem}
            \label{brunello}
            Let be $\eta+\frac{\beta}{m}+1 >0$, $m\in \mathbb{N}$, we have
            that
            \begin{equation}
                I_m^{\eta,\alpha}x^{\beta}=\frac{\Gamma\left(\eta+\frac{\beta}{m}+1\right)}
                {\Gamma\left(\alpha+\eta+1+\frac{\beta}{m}\right)}x^{\beta}.
            \end{equation}
        \end{lem}
        \begin{proof}
            It suffices to calculate the Erd\'elyi--Kober integral
            \begin{align*}
                I^{\eta, \alpha}_{m}x^{\beta}=\frac{x^{-m\eta-m\alpha}}{\Gamma(\alpha)}\int_0^x
                \left(x^m-u^m\right)^{\alpha-1}u^{m\eta}u^{\beta}\, d(u^m).
            \end{align*}
        \end{proof}
        We are now ready to state the following
        \begin{te}
            \label{gallico}
            Let $\alpha\in(0,1]$. The fractional equation
            \begin{align*}
                (L_B)^{\alpha}u_{\alpha}(w)= \frac{\lambda^2}{c^{2\alpha}}u_{\alpha}(w),
            \end{align*}
            is satisfied by
            \begin{equation}
                \label{k5}
                u_{\alpha}(w)=w^{2\alpha-2}\sum_{k=0}^{\infty}\left(
                \frac{\lambda}{2^\alpha c^{\alpha}}w^{\alpha}\right)^{2k}
                \frac{1}{[\Gamma(\alpha k+\alpha)]^2}.
            \end{equation}
        \end{te}
        \begin{proof}
            From \eqref{k3}, we have that
            \begin{align}
                \label{pri}
                (L_B)^{\alpha}w^{\beta} & =4^{\alpha}w^{-2\alpha}I_2^{0,-\alpha}I_2^{0,-\alpha}w^{\beta}\\
                \nonumber & = (\text{by \eqref{mc2}})=
                4^{\alpha}w^{-2\alpha}\left[(1-\alpha)I_2^{0, 1-\alpha}+\frac{1}{2}
                I_2^{0, 1-\alpha}\left(w\frac{d}{dw}\right)\right]^2 w^{\beta}\\
                \nonumber & = 4^{\alpha}w^{-2\alpha}\left(1-\alpha
                +\frac{1}{2}\beta\right)^2 I_2^{0,1-\alpha}I_2^{0,1-\alpha}w^{\beta}\\
                \nonumber & = (\text{by lemma (3.1))}=4^{\alpha}w^{-2\alpha}\left(1
                -\alpha+\frac{1}{2}\beta\right)^2\left[\frac{\Gamma\left(\frac{\beta}{2}+1\right)}
                {\Gamma\left(1-\alpha+1+\frac{\beta}{2}\right)}\right]^2 w^{\beta}\\
                \nonumber & = 4^{\alpha}\left[\frac{\Gamma\left(\frac{\beta}{2}+1\right)}
                {\Gamma\left(1-\alpha+\frac{\beta}{2}\right)}\right]^2 w^{\beta-2\alpha}.
            \end{align}
            By applying now the operator \eqref{k3} to the function \eqref{k5}
            we have that (being $\beta = 2\alpha k+2\alpha-2$)
            \begin{align}
                &(L_B)^{\alpha}\left(w^{2\alpha-2}\sum_{k=0}^{\infty}
                \left(\frac{\lambda}{2^\alpha c^{\alpha}}w^{\alpha}\right)^{2k}
                \frac{1}{[\Gamma(\alpha k+\alpha)]^2}\right)
                =4^{\alpha}\sum_{k=0}^{\infty}\left(\frac{\lambda}{2^\alpha c^{\alpha}}\right)^{2k}
                \frac{w^{2\alpha k-2}}{[\Gamma(\alpha k)]^2}\\
                \nonumber & =\frac{\lambda^2}{c^{2\alpha}}
                \sum_{k'=0}^{\infty}\left(\frac{\lambda}{2^\alpha c^{\alpha}}\right)^{2k}
                \frac{w^{2\alpha k'+2\alpha-2}}{[\Gamma(\alpha k'+\alpha)]^2}
                = \frac{\lambda^2}{c^{2\alpha}}u_{\alpha}(w),
            \end{align}
            where $k'+1=k$.
        \end{proof}

        \begin{os}
            We observe that
            \begin{align*}
                [(L_B)^{\alpha}]^n u_{\alpha}(w)=\underbrace{(L_B)^{\alpha}\dots (L_B)^{\alpha}}_{\text{n}}
                u_{\alpha}(w)=
                \frac{\lambda^{2n}}{c^{2n\alpha}}u_{\alpha}(w), \qquad \alpha \in (0,1],
            \end{align*}
            by simply iterating the result of Theorem \ref{gallico}, is satisfied by
            the function \eqref{k5}.
        \end{os}

        \begin{os}
            The solution \eqref{k5} of equation \eqref{k2} can be expressed
            in terms of generalized $\beta$-Mittag--Leffler functions \citep{GP}
            \begin{equation}
                \label{gm}
                E_{\beta;\nu,\gamma}(x)=\sum_{k=0}^{\infty}\frac{x^k}{[\Gamma(\nu k+\gamma)]^\beta},
                \qquad \beta>0, \: \nu>0, \: \gamma\in \mathbb{R}.
            \end{equation}
            The function \eqref{gm}, for $\nu = \gamma = 1$ and $\beta \in
            \mathbb{N}$, coincides with the hyper-Bessel function \citep[see, for example,][]{ya}
            \begin{equation}
                \label{gm1}
                E_{n;1,1}(x)=\sum_{k=0}^{\infty}\frac{x^k}{(k!)^{n}}=I_{0,n}(n\sqrt[n]{x}).
            \end{equation}
            The solution \eqref{k5} can be also represented in terms of multi-index
            Mittag--Leffler functions, defined as \citep{ky}
            \begin{equation}
                E^{(m)}_{(\rho_i), (\mu_i)}(x)=\sum_{k=0}^{\infty}\frac{x^k}{\prod_{j=1}^m\Gamma(k\rho_j+\mu_j)},
                \qquad m\in\mathbb{N}, \: \rho_1, \dots, \rho_m>0, \: \mu_1,\dots,\mu_m \in \mathbb{R}.
            \end{equation}
            Thus the solution \eqref{k5} can be written as
            \begin{equation}
                u_{\alpha}(w)= w^{2\alpha-2}E_{2;\alpha,\alpha}
                \left(\frac{\lambda^2 w^{2\alpha}}{2^{2\alpha} c^{2\alpha}}\right).
            \end{equation}
        \end{os}

        \begin{os}
            It is simple to prove that the function \eqref{k5} written in terms of the variables $z_1$ and $z_2$, i.e.\
            \begin{align*}
                u_{\alpha}(z_1,z_2)=\sum_{k=0}^{\infty}\left(\frac{\lambda}{2^\alpha c^{\alpha}}\right)^{2k}
                \frac{(z_1z_2)^{\alpha k+\alpha-1}}{[\Gamma(\alpha k+\alpha)]^2},
            \end{align*}
            is a solution of the equation
            \begin{equation}
                (4c^2)^{\alpha}\frac{\partial^{\alpha}}{\partial z_1^{\alpha}}
                \frac{\partial^{\alpha}}{\partial z_2^{\alpha}}u_{\alpha}(z_1,z_2)= \lambda^2 u_{\alpha}(z_1,z_2),
            \end{equation}
            where the partial fractional derivatives $\partial^{\alpha}/\partial z_j^{\alpha}$ are
            in the sense of Riemann--Liouville \citep[][Section 2.3]{pod}.
            This result suggests the validity of the following equality:
            \begin{align*}
                \left(4c^2\frac{\partial}{\partial z_1}
                \frac{\partial}{\partial z_2}\right)^{\alpha}=(4c^2)^{\alpha}
                \frac{\partial^{\alpha}}{\partial z_1^{\alpha}}
                 \frac{\partial^{\alpha}}{\partial z_2^{\alpha}}.
            \end{align*}
        \end{os}

        Going back to the original problem, the equation \eqref{KG}
        admits the solution
        \begin{equation}
            \label{tel}
            u_{\alpha}(x,t)=\left(c^2t^2-x^2\right)^{\alpha-1}E_{2;\alpha,\alpha}
            \left(\frac{\lambda^2}{2^{2 \alpha}c^{2\alpha}}
            \left(c^2t^2-x^2\right)^{\alpha}\right)
        \end{equation}
        and for $\alpha =1$, in view of \eqref{gm1}, we have that
        \begin{align}
            u_1(x,t)&=E_{2;1,1}\left(\frac{\lambda^2}{4 c^{2}}
            \left(c^2t^2-x^2\right)\right)
            =\sum_{k=0}^{\infty}\left(\frac{\lambda}{2c}\sqrt{c^2t^2-x^2}\right)^{2k}
            \frac{1}{(k!)^2}
            = I_0\left(\frac{\lambda}{c}\sqrt{c^2t^2-x^2}\right).
        \end{align}

        \begin{os}
            In the case of the fractional Klein--Gordon equation,
            corresponding to \eqref{KG}, i.e.\
            \begin{equation}
                \left(\frac{\partial^2}{\partial t^2}-c^2\frac{\partial^2}{\partial
                x^2}\right)^{\alpha}u_{\alpha}(x,t)=-\lambda^2 u_\alpha(x,t),
            \end{equation}
            the solution can be written as
            \begin{equation}
                u_{\alpha}(x,t)=(c^2t^2-x^2)^{\alpha-1}\sum_{k=0}^{\infty}(-1)^k
                \frac{\lambda^{2k}}{(2c)^{2\alpha k}}\left(c^2t^2-x^2\right)^{\alpha
                k},
            \end{equation}
            and for $\alpha =1$, reduces to
            \begin{align*}
                u_1(x,t)=J_0\left(\frac{\lambda}{c}\sqrt{c^2t^2-x^2}\right), \qquad |x|<ct.
            \end{align*}
        \end{os}

        Let us now introduce a further analytical result which will be used in Section \ref{fractiol} to construct
        a stochastic process related to the fractional Klein--Gordon equation.

        \begin{te}
            \label{otrebor}
            The function
            \begin{align}
                \label{sangemini}
                F(x,t) & = ct \sum_{k=1}^\infty \left( \frac{\lambda}{2^\alpha c^\alpha} \right)^{2k}
                \frac{(c^2 t^2-x^2)^{\alpha k-1}}{\Gamma(\alpha k+1) \Gamma(\alpha k)}
                = \frac{1}{2c} \frac{\partial}{\partial t} \sum_{k=1}^\infty
                \left( \frac{\lambda}{2^\alpha c^\alpha} \right)^{2k} \frac{(c^2 t^2 -x^2)^{\alpha k}}{
                \left[ \Gamma(\alpha k +1) \right]^2}
            \end{align}
            solves the fractional Klein--Gordon equation \eqref{KG}.
        \end{te}

        \begin{proof}
            Let
            \begin{align*}
                G(w) = \sum_{k=1}^\infty
                    \left( \frac{\lambda}{2^\alpha c^\alpha} \right)^{2k} \frac{w^{2 \alpha k}}{
                    \left[ \Gamma(\alpha k +1) \right]^2}, \qquad w = \sqrt{c^2t^2-x^2}.
            \end{align*}
            By using \eqref{pri}, we have that
            \begin{align}
                \label{l1}
                (L_B)^{\alpha}G(w) = \left( \frac{\lambda}{c^\alpha} \right)^2 \left[ G(w) +1 \right].
            \end{align}
            By passing from the variable $w$ to $(x,t)$, we have from
            \eqref{l1} that
            \begin{equation}
                \left(\frac{\partial^2}{\partial t^{2}}-c^2\frac{\partial^2}{\partial x^2}\right)^{\alpha}G(x,t)
                = \lambda^2 \left[G(x,t)+1\right],
            \end{equation}
            and thus
            \begin{align}
                \label{com}
                \frac{\partial}{\partial t}
                \left(\frac{\partial^2}{\partial t^{2}}-c^2\frac{\partial^2}{\partial x^2}\right)^{\alpha}G(x,t)
                = \lambda^2 \frac{\partial}{\partial t}G(x,t).
            \end{align}

            We now show that
            \begin{align*}
                \frac{\partial}{\partial t}\left(\frac{\partial^2}{\partial t^2}-c^2
                \frac{\partial^2}{\partial x^2}\right)^{\alpha}
                G(x,t)=\left(\frac{\partial^2}{\partial t^2}-c^2\frac{\partial^2}{\partial x^2}
                \right)^{\alpha}
                \frac{\partial}{\partial t}G(x,t).
            \end{align*}
            For
            \begin{equation}
                \nonumber
                \begin{cases}
                    z_1=ct+x,\\
                    z_2=ct-x,
                \end{cases}
            \end{equation}
            we have that
            \begin{align*}
                \frac{\partial}{\partial t} = c\frac{\partial}{\partial z_1}+c
                \frac{\partial}{\partial z_2},
                \qquad \frac{\partial}{\partial x} = \frac{\partial}{\partial z_1}-
                \frac{\partial}{\partial
                z_2},
            \end{align*}
            and therefore
            \begin{align*}
                \left(\frac{\partial^2}{\partial t^2}-c^2\frac{\partial^2}{\partial x^2}\right)^{\alpha}
                = \left(4c^2 \frac{\partial^2}{\partial z_1 \partial z_2}\right)^{\alpha}.
            \end{align*}
            Hence
            \begin{align}
                \nonumber &\left[\frac{\partial}{\partial t}\left(\frac{\partial^2}{\partial t^2}-
                c^2\frac{\partial^2}{\partial x^2}\right)^{\alpha}-
                \left(\frac{\partial^2}{\partial t^2}-c^2\frac{\partial^2}{\partial x^2}\right)^{\alpha}
                \frac{\partial}{\partial t}\right]G(x,t)\\
                \nonumber & = 4^{\alpha}c^{2\alpha+1}\left[
                \left(\frac{\partial}{\partial z_1}+\frac{\partial}{\partial z_2}\right)
                \left(\frac{\partial}{\partial z_1}\frac{\partial}{\partial z_2}\right)^{\alpha}
                -\left(\frac{\partial}{\partial z_1}\frac{\partial}{\partial z_2}\right)^{\alpha}
                \left(\frac{\partial}{\partial z_1}+\frac{\partial}{\partial z_2}\right)\right]G(z_1,z_2).
            \end{align}
            It is simple to prove that
            \begin{align}
                \nonumber \frac{\partial}{\partial z_1}\frac{\partial^{\alpha}}{\partial z_1^{\alpha}}
                \frac{\partial^{\alpha}}{\partial z_2^{\alpha}}G(z_1,z_2)=
                \frac{\partial^{\alpha}}{\partial z_1^{\alpha}}\frac{\partial^{\alpha}}{\partial z_2^{\alpha}}
                \frac{\partial}{\partial z_1}G(z_1,z_2)
                = \frac{1}{z_1} \sum_{k=1}^{\infty}\left(\frac{\lambda}{2^\alpha c^\alpha}\right)^{2k}
                \frac{(z_1 z_2)^{\alpha k - \alpha}}{\Gamma(\alpha k - \alpha) \Gamma(\alpha k -\alpha +1)},\\
                \nonumber \frac{\partial}{\partial z_2}
                \frac{\partial^{\alpha}}{\partial z_1^{\alpha}}\frac{\partial^{\alpha}}
                {\partial z_2^{\alpha}}G(z_1,z_2)=\frac{\partial^{\alpha}}{\partial z_1^{\alpha}}
                \frac{\partial^{\alpha}}{\partial z_2^{\alpha}}\frac{\partial}{\partial z_2}G(z_1,z_2)
                =\frac{1}{z_2} \sum_{k=1}^{\infty}\left(\frac{\lambda}{2^\alpha c^\alpha}\right)^{2k}
                \frac{(z_1 z_2)^{\alpha k - \alpha}}{\Gamma(\alpha k - \alpha) \Gamma(\alpha k -\alpha +1)},
            \end{align}
            where the partial fractional derivatives are in the sense of Riemann--Liouville (see Remark 3.5).
            We have just shown in fact that
            \begin{align*}
                \left[\frac{\partial}{\partial t}\left(\frac{\partial^2}{\partial t^2}-c^2
                \frac{\partial^2}{\partial x^2}\right)^{\alpha}
                -\left(\frac{\partial^2}{\partial t^2}-c^2\frac{\partial^2}{\partial x^2}\right)^{\alpha}
                \frac{\partial}{\partial t}\right]G(x,t)=0.
            \end{align*}
            Returning to \eqref{com}, we have
            \begin{equation}
                \nonumber
                \frac{\partial}{\partial t}\left(\frac{\partial^2}{\partial t^2}-c^2
                \frac{\partial^2}{\partial x^2}\right)^{\alpha}
                G(x,t)=\left(\frac{\partial^2}{\partial t^2}-c^2\frac{\partial^2}{\partial x^2}\right)^{\alpha}
                \frac{\partial}{\partial t}G(x,t)=
                \lambda^2 \frac{\partial}{\partial t}G(x,t).
            \end{equation}
            Being
            \begin{align*}
                \frac{\partial}{\partial t}G(x,t)=2c \, F(x,t),
            \end{align*}
            we finally arrive at
            \begin{equation}
                \left(\frac{\partial^2}{\partial t^2}-c^2\frac{\partial^2}{\partial x^2}\right)^{\alpha}F(x,t)
                = \lambda^2 F(x,t).
            \end{equation}
        \end{proof}

        \begin{os}
            We note that for $\alpha =1$
            \begin{align*}
                F(x,t)=\frac{1}{2c} \frac{\partial}{\partial t}I_0\left(\frac{\lambda}{c}\sqrt{c^2t^2-x^2}\right),
                \qquad |x|<ct.
            \end{align*}
        \end{os}

        \begin{os}
            We observe that in general
            \begin{align*}
                \frac{\partial}{\partial t}\left(\frac{\partial^2}{\partial t^2}-c^2\frac{\partial^2}{\partial x^2}
                \right)^{\alpha}f(x,t)\neq \left(\frac{\partial^2}{\partial t^2}-c^2\frac{\partial^2}{\partial x^2}
                \right)^{\alpha}\frac{\partial}{\partial t}f(x,t).
            \end{align*}
            Indeed, given a certain function $f(z_1,z_2)$,
            \begin{align*}
                \frac{\partial}{\partial z_1}\frac{\partial^{\alpha}}{\partial z_1^{\alpha}}
                \frac{\partial^{\alpha}}{\partial z_2^{\alpha}}f(z_1,z_2)\neq
                \frac{\partial^{\alpha}}{\partial z_1^{\alpha}}
                \frac{\partial^{\alpha}}{\partial z_2^{\alpha}}\frac{\partial}{\partial z_1}f(z_1,z_2),
            \end{align*}
            this is due to the fact that the fractional derivatives of order $\alpha \in (0,1)$
            do not commute in general with the ordinary derivatives (see e.g. \citet[][Section 2.3.5]{pod}):
            \begin{align*}
                \frac{d}{dz_1}\frac{d^{\alpha}}{dz_1^{\alpha}}f(z_1)=
                \frac{d^{\alpha}}{d z_1^{\alpha}}\frac{d}{d z_1}f(z_1)+\frac{z_1^{-\alpha}}{\Gamma(1-\alpha)}f(z_1)
                \bigg|_{z_1=0}.
            \end{align*}
        \end{os}

    \section{Fractional telegraph-type processes}

        \label{fractiol}
        The classical symmetric telegraph process is defined as
        \begin{equation}
            \mathcal{T}(t)=V(0)\int_0^t(-1)^{\mathcal{N}(s)}ds, \qquad t \ge 0,
        \end{equation}
        where $V(0)$ is a two-valued random variable independent of
        the Poisson process $\mathcal{N}(t)$, $t\geq 0$. The telegraph
        process is a finite-velocity random motion where changes of
        direction are governed by the homogeneous Poisson process
        $\mathcal{N}(t)$.

        It is well-known that \citep[see, for example,][and
        the references therein]{ale}

        \begin{equation}
            P\{\mathcal{T}(t)\in dx|\mathcal{N}(t)=2k+1\}=
            dx\frac{(2k+1)!}{(k!)^2}\frac{\left(c^2t^2-x^2\right)^k}{(2ct)^{2k+1}},
            \qquad k\geq 0, \: |x|< ct,
            \label{cot}
        \end{equation}
        \begin{equation}
            P\{\mathcal{T}(t)\in dx|\mathcal{N}(t)=2k\}=
            dx\frac{ct(2k)!}{k!(k-1)!}\frac{\left(c^2t^2-x^2\right)^{k-1}}{(2ct)^{2k}},
            \qquad k\geq 1, \: |x|< ct,
            \label{cot1}
        \end{equation}

        \begin{align}
            &P\{\mathcal{T}(t)\in dx\}=dx \frac{e^{-\lambda t}}{2c}\left[\lambda I_0\left(\frac{\lambda}{c}
            \sqrt{c^2t^2-x^2}\right)+\frac{\partial}{\partial t}I_0\left(\frac{\lambda}{c}
            \sqrt{c^2t^2-x^2}\right)
            \right], \qquad  |x|<ct,
            \label{tp}
        \end{align}

        \begin{equation}
            P\{\mathcal{T}(t)=\pm ct\}=\frac{e^{-\lambda t}}{2}.
            \label{diso1}
        \end{equation}

        The absolutely continuous component of the distribution of the
        telegraph process \eqref{tp} is the solution to the Cauchy
        problem
        \begin{equation}
            \label{te1}
            \begin{cases}
                \frac{\partial^2 p}{\partial t^2}+2\lambda \frac{\partial p}{\partial
                t}= c^2 \frac{\partial^2 p}{\partial x^2},\\
                p(x,0)=\delta(x),\\
                \frac{\partial p}{\partial t}(x,t)\bigg|_{t=0}=0.
            \end{cases}
        \end{equation}
        By means of the transformation $p(x,t)= e^{-\lambda t}u(x,t)$,
        equation \eqref{te1} is converted into the Klein--Gordon equation
        \eqref{KG} for $\alpha =1$.

        Our aim here is to construct a fractional generalization of the
        telegraph process whose absolutely continuous component of its
        distribution is related to the fractional Klein--Gordon equation
        \eqref{KG}.
        We first recall the fractional Poisson process, $\mathcal{N}^{\alpha}(t)$, $t\geq 0$,
        whose one-dimensional distribution has the following form
        \begin{align}
            \label{po}
            P\{\mathcal{N}^{\alpha}(t)=k\}= \frac{1}{E_{\alpha, 1}
            (\lambda t^{\alpha})}\frac{\left(\lambda t^{\alpha}\right)^k}{\Gamma(\alpha k +1)},
            \qquad \alpha \in (0,1], \: k\geq 0.
        \end{align}
        Such fractional Poisson process was first discussed by \citet{Beghin}, where the distribution
        slightly differs from \eqref{po}.
        The probability generating function of \eqref{po} reads
        \begin{equation}
            \label{gen}
            G_{\alpha}(u,t)=\frac{E_{\alpha, 1}(u\lambda t^{\alpha})}{E_{\alpha,1}(\lambda
            t^{\alpha})}, \qquad |u|<1.
        \end{equation}
        In \citet{Bala}, a general class of weighted Poisson processes has
        been introduced, of which \eqref{po} is a special case since
        \begin{equation}
        P\{\mathcal{N}^{\alpha}(t)=k\}=\frac{\frac{k!}{\Gamma(\alpha k+1)}P\{\mathcal{N}
        (t^{\alpha})=k\}}{\sum_{j=0}^{\infty}\frac{j!}{\Gamma(\alpha j+1)}P\{\mathcal{N}
        (t^{\alpha})=j\}}.
        \end{equation}

        Fractionality of \eqref{po} is due to
        the fact that \eqref{gen} solves the fractional equation
        \begin{align*}
            \frac{{}^C \partial^{\alpha}}{\partial u^{\alpha}}G_{\alpha}(u^{\alpha},t)
            =\lambda t^{\alpha}G_{\alpha}(u^{\alpha},t),
        \end{align*}
        where ${}^C \partial^\alpha/ \partial u^\alpha$ is the so-called Caputo fractional derivative
        \citep[][Section 2.4]{pod}.
        We note that
        \begin{align}
            \label{mlr}
            &\sum_{k=0}^{\infty} P\{\mathcal{N}^{\alpha}(t)=k\}=
            \sum_{k=0}^{\infty} P\{\mathcal{N}^{\alpha}(t)=2k+1\}+\sum_{k=0}^{\infty} P\{\mathcal{N}^{\alpha}(t)=2k\}\\
            \nonumber &= \frac{\lambda t^{\alpha}E_{2\alpha,\alpha + 1}(\lambda^2t^{2\alpha})}{E_{\alpha, 1}
            (\lambda t^{\alpha})}
            + \frac{E_{2\alpha, 1}(\lambda^2 t^{2\alpha})}{E_{\alpha, 1}(\lambda t^{\alpha})},
        \end{align}
        such that, for $\alpha = 1$, we have
        \begin{align}
            \nonumber \sum_{k=0}^{\infty}P\{\mathcal{N}(t)=k\}= e^{-\lambda t}\left(\sinh(\lambda t)
            +\cosh(\lambda t)\right)=1.
        \end{align}

        The solution \eqref{sangemini} can be written as
        \begin{equation}
            F(x,t) \, dx = E_{\alpha, 1}(\lambda t^{\alpha})\sum_{k=1}^{\infty}
            P\{\mathcal{T}^{\alpha}(t)\in dx| \mathcal{N}^{\alpha}(t)=2k\}P\{\mathcal{N}^{\alpha}(t)=2k\},
        \end{equation}
        where
        \begin{equation}
            \label{con}
            P\{\mathcal{T}^{\alpha}(t)\in dx| \mathcal{N}^{\alpha}(t)=2k\}
            =dx \frac{\left(c^2t^2-x^2\right)^{\alpha k -1}}{(2ct)^{2k\alpha-1}}
            \frac{\Gamma(2\alpha k)}{\left[\Gamma(\alpha k)\right]^2},\qquad k\geq 1,\:
            |x|<ct,
        \end{equation}
        and $P\{\mathcal{N}^{\alpha}(t)=2k\}$ is given by \eqref{po}. We used the symbol $\mathcal{T}^{\alpha}(t)$
        in order to consider a fractional-type generalization of the telegraph process
        that includes for $\alpha =1$ the classical one.
        The conditional densities \eqref{con} can be found as the
        laws of the r.v.'s
        \begin{equation}
            \label{os}
            \mathcal{T}^{\alpha}(t)= ct\left[T^{\alpha}_{(n^+)}-(1-T^{\alpha}_{(n^+)})\right],
        \end{equation}
        where $T^{\alpha}_{(n^+)}$ possesses probability
        density given by
        \begin{equation}
            \label{os1}
            f_{T^{\alpha}_{(n^+)}}(w) = \frac{\Gamma(n\alpha)}
            {\Gamma(n^+\alpha)\Gamma((n-n^+)\alpha)} w^{n^+\alpha-1}(1-w)^{(n-n^+)\alpha-1}, \qquad 0<w<1.
        \end{equation}
        The r.v.\ defined in \eqref{os} can be regarded as a rightward
        displacement of random length of $ct \, T^{\alpha}_{(n^+)}$
        and a leftward displacement for the remaining interval of
        time.

        If $\alpha=1$, the r.v.'s $\mathcal{T}^{1}(t)$ coincides
        in distribution with
        \begin{equation}
            \label{ex}
            \mathcal{T}^1(t) \overset{\text{d}}{=} ct\left[T_1-(T_2-T_1)+\dots+(-1)^{n+1}(t-T_n)\right],
        \end{equation}
        where $t\, T_k$, $k=1,\dots,n$ are the instants where the Poisson
        events of $\mathcal{N}^1(t)$ occur (see \citet{ale}). In force of the
        exchangeability of the r.v.'s $T_1,\dots, T_n$, we can establish the following equality in distribution
        \begin{align*}
            \mathcal{T}^1(t)
            \overset{\text{d}}{=} ct\left[T_{(n^+)}-(1-T_{(n^+)})\right].
        \end{align*}
        Note however, that in the fractional case a similar equality in
        distribution cannot be established. This is due to the fact
        that we do not have the multivariate distribution
        $\left(\mathcal{T}^{\alpha}_1, \mathcal{T}^{\alpha}_2,\dots, \mathcal{T}^{\alpha}_n\right)$,
        where $\mathcal{T}^{\alpha}_j$ are the instants of occurrence of the
        events of the fractional Poisson process.\\
        The distribution of \eqref{os} coincides with
        \eqref{con}. Indeed,
        \begin{align}
            \label{cond}
            P\{\mathcal{T}^{\alpha}(t)\in dx|\mathcal{N}^{\alpha}(t)=2k, V(0)\}
            &=\frac{d}{dx}P\Bigl\{\mathcal{T}^{\alpha}_{(n^+)}<\frac{ct+x}{2ct}\Bigr|\mathcal{N}^{\alpha}(t)
            =2k, V(0)\Bigr\} \,dx\\
            \nonumber &=dx \frac{\Gamma(2k\alpha)}
            {(2ct)[\Gamma(\alpha k)]^2}\left[\left(\frac{ct+x}{2ct}\right)^{k\alpha-1}
            \left(\frac{ct-x}{2ct}\right)^{k\alpha-1}\right]\\
            \nonumber &=dx \frac{\left(c^2t^2-x^2\right)^{\alpha k -1}}{(2ct)^{2k\alpha-1}}
            \frac{\Gamma(2\alpha k)}{\left[\Gamma(\alpha k)\right]^2}, \qquad k \ge 1.
        \end{align}
        A similar approach can be adopted when the
        fractional Poisson process $\mathcal{N}^{\alpha}(t)$ takes an
        odd number of events. In this case the displacement \eqref{os}
        involves the r.v.\ $T^{\alpha}_{(n^+)}$ with
        density
        \begin{equation}
            \label{samsung}
            f_{T^{\alpha}_{(n^+)}}(w) = \frac{\Gamma(n\alpha+1)}
            {\Gamma(n^+\alpha+\frac{1-\alpha}{2})\Gamma((n-n^+)\alpha+\frac{1-\alpha}{2}+\alpha)}
            w^{n^+\alpha+\frac{1-\alpha}{2}-1}
            (1-w)^{(n-n^+)\alpha+\frac{1-\alpha}{2}+\alpha -1},\quad w\in(0,1).
        \end{equation}
        The conditional distribution when $\mathcal{N}^{\alpha}(t)=2k+1$
        reads
        \begin{align}
            \label{cond2}
            &P\{\mathcal{T}^{\alpha}(t)\in dx|\mathcal{N}^{\alpha}(t)=2k+1, V(0)\} \\
            \nonumber &=\frac{d}{dx}P\Bigl\{\mathcal{T}^{\alpha}_{(n^+)}<\frac{ct+x}{2ct}\Bigr| \mathcal{N}^{\alpha}(t)
            =2k+1, V(0)\Bigr\} \, dx\\
            \nonumber &=dx \frac{\Gamma(2k\alpha+\alpha+1)}
            {(2ct)[\Gamma(\alpha k+\alpha+\frac{1-\alpha}{2})]^2}\left[\left(\frac{ct+x}{2ct}\right)^{k\alpha+
            \alpha+\frac{1-\alpha}{2}-1}
            \left(\frac{ct-x}{2ct}\right)^{k\alpha+\alpha+\frac{1-\alpha}{2}-1}\right]\\
            \nonumber &=dx\frac{\left(c^2t^2-x^2\right)^{\alpha k +\alpha+\frac{1-\alpha}{2}-1}}{(2ct)^{2k\alpha+\alpha}}
            \frac{\Gamma(2\alpha k +\alpha+1)}{[\Gamma(\alpha k+\alpha+
            \frac{1-\alpha}{2})]^2}, \qquad k \ge 0.
        \end{align}
        We observe that for $\alpha=1$, the conditional distributions
        \eqref{cond} and \eqref{cond2} coincide with \eqref{cot1} and \eqref{cot},
        respectively.
        Clearly \eqref{samsung} for $\alpha =1$ coincides with the
        distribution of the $n^+$-th order statistics related to the occurrence
        of Poissonian events, while for
        $0<\alpha<1$, is a Beta random variable. Furthermore,
        note that also \eqref{os1} reduces to the
        distribution of the $n^+$-th order statistics but in this case
        we have to consider a number $n+1$ of random variables.

        In view of all these results we arrive at
        the following statement.
        \begin{te}
            The fractional telegraph-type process $\mathcal{T}^{\alpha}(t)$,
            $t\geq 0$, has the following probability law
            \begin{align}
                \label{dis}
                p^{\alpha}(x,t)= {} &\frac{1}{E_{\alpha, 1}(\lambda t^{\alpha})}\bigg[ct \sum_{k=1}^{\infty}
                \left(\frac{\lambda}{2^\alpha c^{\alpha}}\right)^{2k}
                \frac{(c^2t^2-x^2)^{\alpha k -1}}{\Gamma(\alpha k)\Gamma(\alpha k +1)}
                + \sum_{k=0}^{\infty}\left(\frac{\lambda}{2^\alpha c^{\alpha}}\right)^{2k+1}
                \frac{(c^2t^2-x^2)^{\alpha k + \frac{\alpha-1}{2}}}{[\Gamma(\alpha k
                + \frac{1+\alpha}{2})]^2}\bigg]\\
                \nonumber & +\frac{1}{2 E_{\alpha, 1}(\lambda
                t^{\alpha})}[\delta(x+ct)+\delta(x-ct)], \qquad \alpha \in
                (0,1], |x|\leq ct.
            \end{align}
        \end{te}
        \begin{proof}
            The singular component of the distribution \eqref{dis}
            \begin{align*}
                p^{\alpha}_s(x,t)=\frac{1}{2 E_{\alpha, 1}(\lambda t^{\alpha})}[\delta(x+ct)+\delta(x-ct)],
            \end{align*}
            is due to the case where no Poisson event occurs up to time $t$
            and the moving particle reaches the endpoints of the interval
            $[-ct,+ct]$ with probability
            \begin{align*}
                P\{\mathcal{N}^{\alpha}(t)=0\}=\frac{1}{E_{\alpha, 1}(\lambda t^{\alpha})}.
            \end{align*}
            The absolutely continuous component of \eqref{dis}
            \begin{align}\label{ramon}
                &p^{\alpha}_{ac}(x,t)=p^{\alpha}(x,t)-p_s^{\alpha}(x,t)\\
                \nonumber &=\sum_{k=1}^{\infty}
                \frac{P\{\mathcal{T}^{\alpha}(t)\in dx|\mathcal{N}^{\alpha}(t)=2k\}P\{\mathcal{N}^{\alpha}(t)=2k\}}{dx}
                \\
                \nonumber & +\sum_{k=0}^{\infty}
                \frac{P\{\mathcal{T}^{\alpha}(t)\in
                dx|\mathcal{N}^{\alpha}(t)=2k+1\}
                P\{\mathcal{N}^{\alpha}(t)=2k+1\}}{dx}
                , \quad |x|<ct, \forall
                t>0.
            \end{align}
            The conditional distributions appearing in \eqref{ramon}
            are given by \eqref{cond} and \eqref{cond2}, while the
            fractional Poisson probabilities
            $P\{\mathcal{N}^{\alpha}(t)=k\}$ are obtained by
            specializing \eqref{po} in the even and odd cases.
        \end{proof}

        From Fig.~1 emerges that for increasing values of
        $\alpha$, the density of $p^{\alpha}_{ac}(x,t)$ behaves as
        that of the telegraph process except near the endpoints $x=
        \pm ct$, where it tends to infinity. Note also that the smaller is $\alpha$,
        the slower the convergence towards
        a bell-shaped form.

        \begin{figure}
            \centering
            \includegraphics[scale=.73]{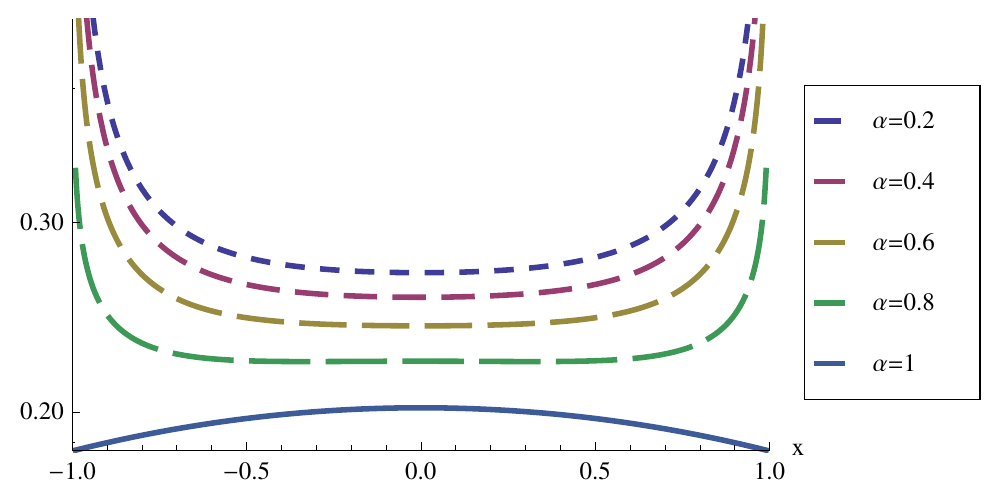}
            \includegraphics[scale=.73]{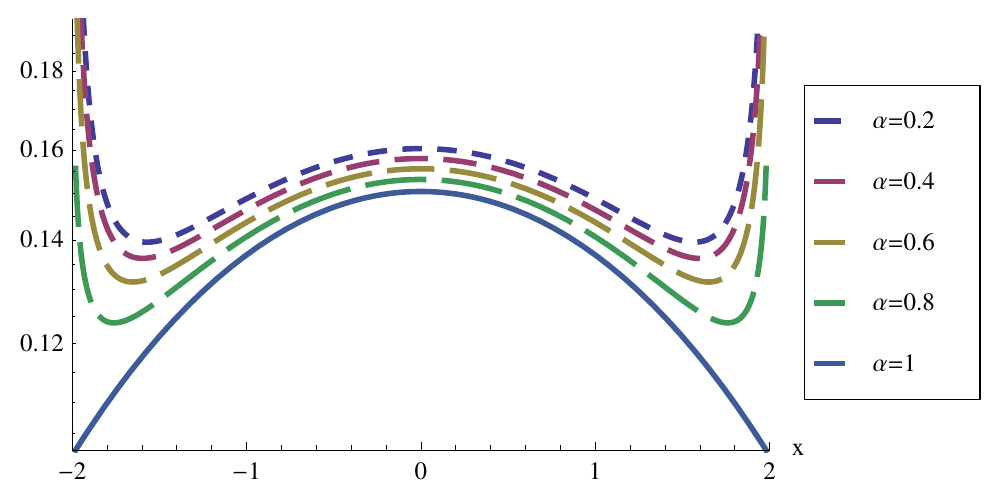}
            \includegraphics[scale=.73]{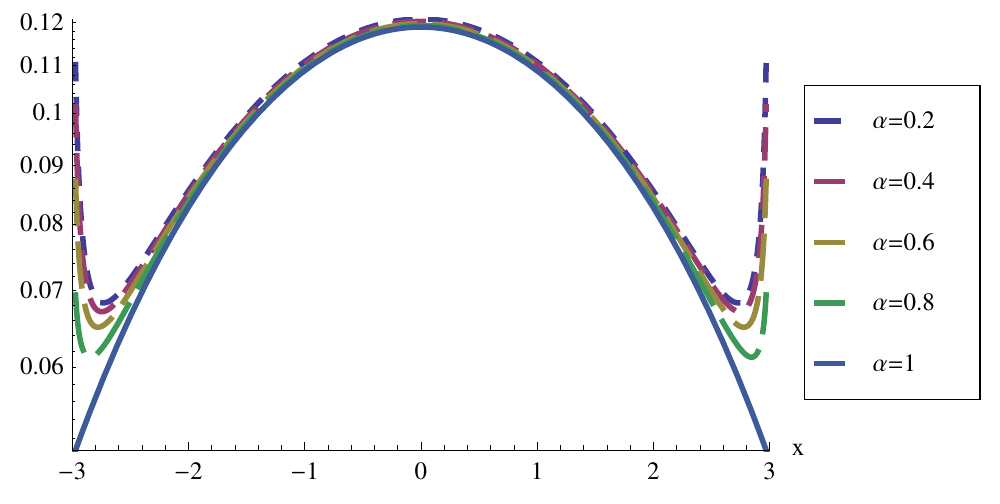}
            \caption{Plot of the absolutely continuous component of the distribution \eqref{dis}
            for different values of $\alpha$, with $(c,\lambda)=(1,1)$ and $t=1$ (top left), $t=2$ (top right)
            $t=3$ (bottom). The figures are in logarithmic scale.}
            \label{figura}
        \end{figure}

        \begin{os}
        Our approach consists in finding solutions of fractional
        Klein-Gordon equations and by means of them construct the
        probability distributions of fractional versions of
        finite-velocity random motions. We can see afterwards that
        the found solutions $p^{\alpha}(x,t)$ satisfy the same
        initial conditions as the classical telegraph process, i.e.
        \begin{equation}\nonumber
            \begin{cases}
                p^{\alpha}(x,0)=\delta(x),\\
                \frac{\partial p^{\alpha}}{\partial t}(x,t)\bigg|_{t=0}=0.
            \end{cases}
        \end{equation}
        \end{os}

        \begin{os}
            For $\alpha =1$, the distribution \eqref{dis} reduces to the sum of
            \eqref{tp} (absolutely continuous component) and \eqref{diso1}
            (singular component).
        \end{os}

        \begin{os}
            For all $k\geq 1$, there exists an order of fractionality $0<\alpha<1$ for which the conditional
            distributions are uniform
            in $[-ct,+ct]$. In particular, for all fixed values of $k \ge 0$, the distribution \eqref{cond2} is uniform
            for
            \begin{align*}
                \alpha = \frac{1}{2k+1},
            \end{align*}
            while for values of $k \ge 1$, the distribution \eqref{cond} is uniform for
            \begin{align*}
                \alpha = \frac{1}{k}.
            \end{align*}

            The densities \eqref{cond2} for
            \begin{align*}
                0 \leq k<\bigg\lfloor\frac{1-\alpha}{2 \alpha}\bigg\rfloor,
            \end{align*}
            and the densities \eqref{cond} for
            \begin{align*}
                1\leq k<\bigg\lfloor\frac{1}{\alpha}\bigg\rfloor,
            \end{align*}
            display an arcsine behaviour, that is, the densities approach to $+\infty$ for $x\rightarrow \pm ct$.
            In the opposite case, they have a bell-shaped form as happens with
            \eqref{cot} and \eqref{cot1} for the classical telegraph process (see Fig.\ \ref{figura}).
            This is an important feature of the fractional telegraph process.
        \end{os}

        \begin{figure}
            \centering
            \includegraphics[scale=.73]{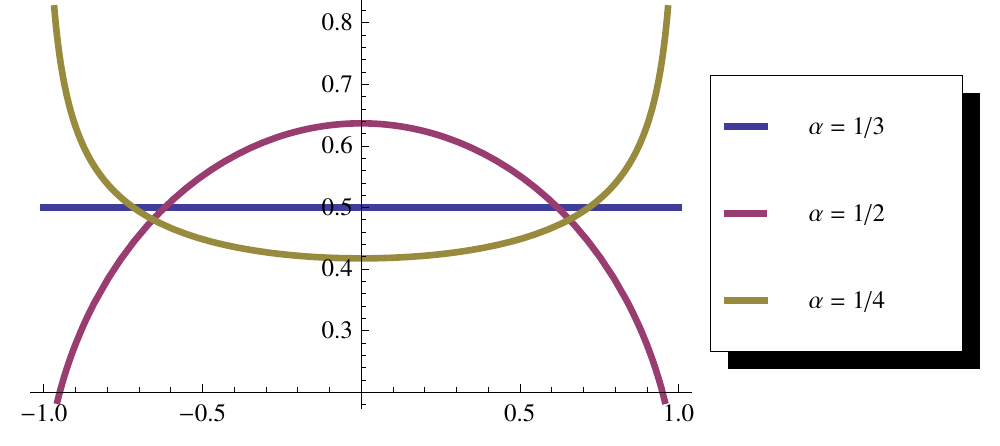}
            \includegraphics[scale=.73]{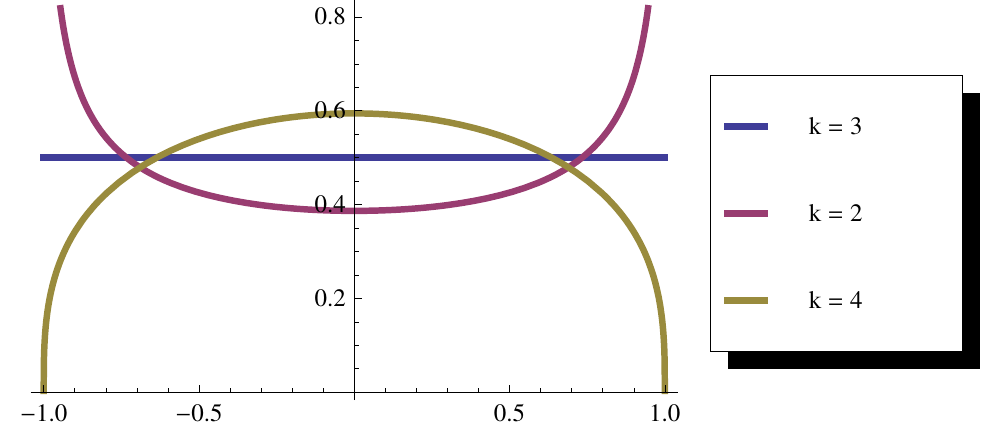}
            \caption{Plot of the conditional density \eqref{con} for $c=1$, $t=1$; Left side:
            $k=3$, $\alpha = (1/3,1/2,1/4)$,
            Right side: $\alpha=1/3$, $k=(3,2,4)$.}
            \label{figura2}
        \end{figure}

        \begin{lem}
            \label{ammel}
            The function
            \begin{align}
                H(x,t) & = E_{\alpha,1}(\lambda t^\alpha) \sum_{k=0}^\infty P \{ \mathcal{T}^\alpha \in dx
                | \mathcal{N}^\alpha (t) = 2k +1 \} P \{ \mathcal{N}^\alpha(t) = 2k+1 \} \\
                & = \sum_{k=0}^{\infty}\left(\frac{\lambda}{2^\alpha c^{\alpha}}\right)^{2k+1}
                \frac{(c^2t^2-x^2)^{\alpha k + \frac{\alpha -1}{2}}}{[\Gamma(\alpha k
                + \frac{1+\alpha}{2})]^2}, \notag
            \end{align}
            is a solution to the non-homogeneous fractional Klein--Gordon equation
            \begin{align}
                \label{nonono}
                \left( \frac{\partial^2}{\partial t^2} -c^2 \frac{\partial^2}{\partial x^2} \right)^\alpha
                u_\alpha(x,t) = \lambda^2 u_\alpha(x,t) + \lambda 2^\alpha c^\alpha
                \frac{(\sqrt{c^2t^2-x^2})^{-\alpha-1}}{[\Gamma(\frac{1-\alpha}{2})]^2}.
            \end{align}
        \end{lem}

        \begin{proof}
            We start by writing the function $H$ in terms of the variable $w = \sqrt{c^2t^2-x^2}$.
            \begin{align*}
                H(w) = \sum_{k=0}^\infty \left( \frac{\lambda}{2^\alpha c^\alpha} \right)^{2k+1}
                \frac{w^{2 \alpha k +\alpha -1}}{[\Gamma(\alpha k + \frac{1+\alpha}{2})]^2}.
            \end{align*}
            By following some steps similar to those of Theorem \ref{gallico} we have that
            \begin{align*}
                (L_B)^\alpha H(w) = \frac{\lambda^2}{c^{2\alpha}} H(w) + \frac{\lambda 2^\alpha}{c^\alpha}
                \frac{w^{-\alpha-1}}{[\Gamma(\frac{1-\alpha}{2})]^2}.
            \end{align*}
            Returning now to the original variables we obtain the claimed result.
        \end{proof}

        We can finally conclude with the following
        \begin{te}
            The function
            \begin{align}
                f(x,t)&= E_{\alpha,1}(\lambda t^{\alpha})\frac{P\{\mathcal{T}^{\alpha}(t)\in dx\}}{dx},
                \qquad x\in (-ct,+ct),\\
                \nonumber & = \bigg[ct \sum_{k=1}^{\infty}
                \left(\frac{\lambda}{2^\alpha c^{\alpha}}\right)^{2k}
                \frac{(c^2t^2-x^2)^{\alpha k -1}}{\Gamma(\alpha k)\Gamma(\alpha k +1)}
                \nonumber + \sum_{k=0}^{\infty}\left(\frac{\lambda}{2^\alpha c^{\alpha}}\right)^{2k+1}
                \frac{(c^2t^2-x^2)^{\alpha k + \frac{\alpha -1}{2}}}{[\Gamma(\alpha k
                + \frac{1+\alpha}{2})]^2}\bigg],
            \end{align}
            where $P\{\mathcal{T}^{\alpha}(t)\in dx\}/dx$ represents the
            absolutely continuous component of the distribution of the
            fractional telegraph process $\mathcal{T}^{\alpha}(t)$, $t\geq 0$,
            is a solution to the non-homogeneous fractional Klein--Gordon equation \eqref{nonono}.
        \end{te}
        \begin{proof}
            The proof is a direct consequence of Theorem \ref{otrebor} and Lemma \ref{ammel}.
        \end{proof}

    \section{Fractional planar random motion at finite velocity}

        A planar random motion at finite velocity with uniformly distributed orientation of displacements has been studied
        by several researchers over the years \citep[see for example][]{sta,kol}.
        The motion is described by a particle taking directions $\theta_j$, $j=1, 2, \dots$, (uniformly
        distributed in $(0,2\pi]$) at Poisson paced times. The orientations $\theta_j$ and the governing Poisson process
        $\mathcal{N}(t)$, $t\geq 0$, are assumed to be independent. The conditional distributions of the current position
        $(X(t),Y(t))$, $t \ge 0$, are given by (see formula (11) of \citet{kol})
        \begin{equation}
            \label{pt1}
            P\{X(t)\in dx, Y(t)\in dy|\mathcal{N}(t)=n \}
            =\frac{n \left(c^2t^2-x^2-y^2\right)^{\frac{n}{2}-1}}{2\pi(ct)^n} dx \, dy,
        \end{equation}
        for $x^2+y^2<c^2t^2$, $n\geq 1$, and possesses characteristic function
        \begin{equation}
            \mathbb{E}\{e^{i\alpha X(t)+i\beta Y(t)}|\mathcal{N}(t)=n\}
            =\frac{2^{n/2}\Gamma(\frac{n}{2}+1)}{\left(ct\sqrt{\alpha^2+\beta^2}\right)^{n/2}}
            J_{\frac{n}{2}}\left(ct\sqrt{\alpha^2+\beta^2}\right), \qquad n\geq 1, \: (\alpha,\beta) \in \mathbb{R}^2.
        \end{equation}
        The unconditional distribution of $(X(t),Y(t))$ reads
        \begin{equation}
            \label{pt}
            P\{X(t)\in dx, Y(t)\in dy\}= \frac{\lambda}{2\pi c}
            \frac{e^{-\lambda t+\frac{\lambda}{c}\sqrt{c^2t^2-x^2-y^2}}}{\sqrt{c^2t^2-x^2-y^2}} dx \, dy,
        \end{equation}
        for $x^2+y^2<c^2t^2$.
        The singular component of $(X(t),Y(t))$ is uniformly distributed on the circumference of radius $ct$
        and has weight $e^{-\lambda t}$.
        It has been proven that the density in \eqref{pt} is a solution to the planar telegraph equation
        (also equation of damped waves)
        \begin{equation}
            \frac{\partial^2 u}{\partial t^2}+2\lambda\frac{\partial u}{\partial t}
            = c^2 \left\{\frac{\partial^2}{\partial x^2}+\frac{\partial^2}{\partial y^2}\right\}u.
        \end{equation}
        In this section we construct a fractional planar random motion at finite velocity whose space-dependent
        component of
        the distribution solves the two-dimensional fractional Klein--Gordon equation
        \begin{equation}
            \label{2d}
            \left(\frac{\partial^2}{\partial t^2}-c^2 \left\{\frac{\partial^2}{\partial x^2}+
            \frac{\partial^2}{\partial y^2})\right\}\right)^{\alpha}u_{\alpha}(x,y,t)
            =\lambda^2 u_{\alpha}(x,y,t),
            \quad \alpha \in (0,1].
        \end{equation}
        By using the transformation
        \begin{align*}
            w = \sqrt{c^2t^2-x^2-y^2},
        \end{align*}
        we convert \eqref{2d} into the Bessel-type fractional equation
        \begin{equation}
            \label{L2d}
            \left(\frac{d^2}{dw^2}+\frac{2}{w}\frac{d}{dw}\right)^{\alpha}u_{\alpha}(w)
            =\frac{\lambda^2}{c^{2\alpha}}u_{\alpha}(w).
        \end{equation}
        Since
        \begin{equation}
            \left(\frac{d^2}{dw^2}+\frac{2}{w}\frac{d}{dw}\right)^{\alpha}
            =\left(\frac{1}{w^2}\frac{d}{dw}w^2\frac{d}{dw}\right)^{\alpha},
        \end{equation}
        this operator coincides with \eqref{L} for $n=2$, $a_1=-2$, $a_2=2$,
        $a_3=0$. Therefore, in view of \eqref{Lo}, $a=0$, $m=2$, $b_1=1/2$ and $b_2=0$, and, in view of
        \eqref{pot}, we can write
        \begin{equation}
            \left(\frac{d^2}{dw^2}+\frac{2}{w}\frac{d}{dw}\right)^{\alpha}u_{\alpha}(w)=4^{\alpha}w^{-2\alpha}
            I_2^{0,-\alpha}I_2^{\frac{1}{2}, -\alpha}u_{\alpha}(w).
        \end{equation}
        We are now ready to state the following
        \begin{te}
            \label{berger}
            A solution to \eqref{L2d} is given by
            \begin{align}
                u_{\alpha}(w)&= \sum_{k=0}^{\infty}\left(\frac{\lambda}{c^{\alpha}}\right)^{2k+2}
                \frac{w^{2\alpha k+2\alpha-2}}{\Gamma(2\alpha k+2\alpha)}
                =\left(\frac{\lambda}{c^{\alpha}}\right)^2 w^{2\alpha-2}
                E_{2\alpha, 2\alpha}\left(\left(\frac{\lambda w^{\alpha}}{c^{\alpha}}\right)^2\right),
                \qquad w \in \mathbb{R}, \: \alpha \in (0,1].
            \end{align}
        \end{te}

        \begin{proof}
            We first observe that
            \begin{align}
                \label{cont}
                \left(\frac{d^2}{dw^2}+\frac{2}{w}\frac{d}{dw}\right)^{\alpha}w^{\beta}&=4^{\alpha}w^{-2\alpha}
                I_2^{0,-\alpha}I_2^{\frac{1}{2}, -\alpha}w^{\beta}=4^{\alpha}w^{-2\alpha}
                I_2^{0,-\alpha}\left(\frac{3}{2}-\alpha+\frac{\beta}{2}\right)\frac{\Gamma(\frac{1}{2}+1+\frac{\beta}{2})}
                {\Gamma(1-\alpha+\frac{1}{2}+\frac{\beta}{2}+1)}w^{\beta}\\
                \nonumber &=4^{\alpha}w^{\beta-2\alpha}
                \left(\frac{3}{2}-\alpha+\frac{\beta}{2}\right)\left(1-\alpha
                +\frac{\beta}{2}\right)\frac{\Gamma(1+\frac{\beta}{2})
                \Gamma(\frac{1}{2}+1+\frac{\beta}{2})}
                {\Gamma(1-\alpha+\frac{1}{2}+\frac{\beta}{2}+1)\Gamma(1-\alpha+\frac{\beta}{2}+1)}\\
                \nonumber &=4^{\alpha}w^{\beta-2\alpha}
                \frac{\Gamma(1+\frac{\beta}{2})
                \Gamma(\frac{1}{2}+1+\frac{\beta}{2})}
                {\Gamma(1-\alpha+\frac{1}{2}+\frac{\beta}{2})\Gamma(1-\alpha+\frac{\beta}{2})}
                =\frac{\Gamma(\beta+2)}{\Gamma(\beta+2-2\alpha)}w^{\beta-2\alpha},
            \end{align}
            where in view of \eqref{mc2} and in force of Lemma \ref{brunello} we used the fact that
            \begin{align}
                \nonumber
                I^{\frac{1}{2}, -\alpha}_2w^{\beta}
                =\left(\frac{3}{2}-\alpha+\frac{\beta}{2}\right)I_2^{\frac{1}{2}, 1-\alpha}w^{\beta}
                = \left(\frac{3}{2}-\alpha+\frac{\beta}{2}\right)\frac{\Gamma(\frac{1}{2}+\frac{\beta}{2}+1)}
                {\Gamma(1-\alpha+\frac{\beta}{2}+\frac{1}{2}+1)}w^{\beta}.
            \end{align}
            In the last step of \eqref{cont} we repeatedly applied the duplication formula of the Gamma function.
            Hence we have that
            \begin{align}
                \nonumber
                &\left(\frac{d^2}{dw^2}+\frac{2}{w}\frac{d}{dw}\right)^{\alpha}\sum_{k=0}^{\infty}
                \left(\frac{\lambda}{c^{\alpha}}\right)^{2k+2}
                \frac{w^{2\alpha k+2\alpha-2}}{\Gamma(2\alpha k+2\alpha)} \\
                & =\sum_{k=0}^{\infty}
                \left(\frac{\lambda}{c^{\alpha}}\right)^{2k+2}
                \frac{w^{2\alpha k-2}}{\Gamma(2\alpha k)}
                \nonumber =\left(\frac{\lambda}{c^{\alpha}}\right)^2\sum_{k'=0}^{\infty}
                \left(\frac{\lambda}{c^{\alpha}}\right)^{2k'+2}
                \frac{w^{2\alpha k'+2\alpha-2}}{\Gamma(2\alpha k'+2\alpha)},
            \end{align}
            as claimed.
        \end{proof}

        \begin{os}
            In light of Theorem \ref{berger} we can state that the two-dimensional fractional Klein--Gordon
            equation \eqref{2d}, admits the following solution:
            \begin{equation}
                \label{plat}
                u_{\alpha}(x,y,t)=\left(\frac{\lambda}{c^{\alpha}}\right)^2\frac{E_{2\alpha,2\alpha}
                \left(\frac{\lambda^2}{c^{2\alpha}}(\sqrt{c^2t^2-x^2-y^2})^{2\alpha}\right)}
                {\left(\sqrt{c^2t^2-x^2-y^2}\right)^{2-2\alpha}}, \qquad x^2+y^2<c^2t^2.
            \end{equation}
            In the specific case $\alpha = 1$, formula \eqref{plat} takes the form
            \begin{equation}
                u_1(x,y,t)=\sum_{k=0}^{\infty}
                \left(\frac{\lambda}{c}\right)^{2k+2}
                \frac{(\sqrt{c^2t^2-x^2-y^2})^{2k}}{(2k+1)!}.
            \end{equation}
        \end{os}

        The solution \eqref{plat} of \eqref{2d} can also be written as
        \begin{equation}
            \nonumber
            u_{\alpha}(x,y,t)=\frac{2\pi c^{2\alpha}}{\lambda^2}E_{\alpha, 1}(\lambda t^{\alpha})\sum_{k=0}^{\infty}
            P\{X^{\alpha}(t)\in dx, Y^{\alpha}(t)\in dy| \mathcal{N}^{\alpha}(t)=2k+2\}P\{\mathcal{N}^{\alpha}(t)=2k+2\},
        \end{equation}
        where
        \begin{align}
            \nonumber & P\{X^{\alpha}(t)\in dx, Y^{\alpha}\in dy | \mathcal{N}^{\alpha}(t)=2k+2\}
            = dx \, dy \frac{2k\alpha+2\alpha}{2\pi(ct)^{2k\alpha+2\alpha}}
            \left(\sqrt{c^2t^2-x^2-y^2}\right)^{2\alpha k +2\alpha-2},
        \end{align}
        for $k\geq 0,\, (x,y)\in C_{ct}$, and where
        \begin{align*}
            C_{ct}=\{(x,y)\in \mathbb{R}^2: x^2+y^2\leq c^2t^2\},
        \end{align*}
        generalizes \eqref{pt1} for even values of $n$.

        If $\mathcal{N}^{\alpha}(t)=2k+1$, we assume that
        \begin{align*}
            & P\{X^{\alpha}(t)\in dx, Y^{\alpha}(t)\in dy| \mathcal{N}^{\alpha}(t)=2k+1\} \\
            & = dy \, dx \, (2k\alpha+\alpha)\left(\sqrt{c^2t^2-(x^2+y^2)}\right)^{2\alpha k +\alpha-2}
            \Big/
            \left( 2\pi(ct)^{2k\alpha+\alpha} \right),
            \qquad k\geq 0,\, (x,y)\in C_{ct}.
        \end{align*}

        With all these preliminaries we arrive at the following theorem:

        \begin{te}
            The vector process $(X^{\alpha}(t), Y^{\alpha}(t))$, $t\geq 0$, $0<\alpha \leq 1$, possesses the following
            probability distribution inside $C_{ct}$
            \begin{align}
                \label{displ}
                P\{X^{\alpha}(t)\in dx, Y^{\alpha}(t)\in dy\}&=\frac{1}{2\pi
                E_{\alpha,1}(\lambda t^{\alpha})}\sum_{k=1}^{\infty}
                \left(\frac{\lambda}{c^{\alpha}}\right)^{k}
                \frac{(c^2t^2-x^2-y^2)^{\alpha \frac{k}{2}-1}}{\Gamma(\alpha
                k)}\\
                \nonumber &=\frac{\lambda}{2\pi c^{\alpha}E_{\alpha,1}(\lambda t^{\alpha})}
                \frac{E_{\alpha,\alpha}\left(
                \frac{\lambda}{c^{\alpha}}\left(\sqrt{c^2t^2-x^2-y^2}\right)^{\alpha}\right)}
                {\left(\sqrt{c^2t^2-x^2-y^2}\right)^{2-\alpha}}.
            \end{align}
        \end{te}

        \begin{proof}
        From the previous calculations, we see that
        \begin{equation}\label{conpl}
        P\{X^{\alpha}(t)\in dx, Y^{\alpha}(t)\in dy| \mathcal{N}^{\alpha}(t)=n\}
            = \frac{\alpha n}{2\pi(ct)^{\alpha n}}\left(c^2t^2-x^2-y^2
                \right)^{\frac{n\alpha}{2}-1}, \quad (x,y)\in C_{ct},
        \end{equation}
        and thus we easily arrive at \eqref{displ}.

        \end{proof}

        Note that \eqref{displ} coincides with \eqref{pt} for $\alpha =1$.

        \begin{os}
            We observe that
            \begin{align}
                \nonumber &\iint_{C_{ct}}P\{X^{\alpha}(t)\in dx, Y^{\alpha}(t)\in dy\}=\frac{1}{2\pi
                E_{\alpha,1}(\lambda t^{\alpha})}
                \sum_{k=1}^{\infty}\left(
                \frac{\lambda}{c^{\alpha}}\right)^k\int_0^{2\pi}\int_0^{ct}\frac{(c^2t^2-\rho^2)^{\frac{\alpha k}{2}-1}
                \rho \, d\rho \, d\theta}{\Gamma(\alpha k)}\\
                \nonumber &= \frac{1}{E_{\alpha,1}(\lambda t^{\alpha})}
                \sum_{k=1}^{\infty}\left(\frac{\lambda}{c^{\alpha}}\right)^k\frac{(ct)^{\alpha k}}{2}
                \int_0^{1}\frac{(1-y)^{\frac{\alpha k}{2}-1}
                dy}{\Gamma(\alpha k)}
                =\frac{1}{E_{\alpha,1}(\lambda t^{\alpha})}
                \sum_{k=1}^{\infty}\lambda^k t^{\alpha k}\frac{\bigg[(1-y)^{\frac{\alpha k}{2}}\bigg]^0_1}{\alpha k
                \Gamma(\alpha k)}\\
                \nonumber &= 1-\frac{1}{E_{\alpha,1}(\lambda t^{\alpha})}= 1-P\{\mathcal{N}^{\alpha}(t)=0\},
            \end{align}
            This can be interpreted by observing that, if no Poisson event occurs,
            the particle performing the fractional planar motion
            arrives at $\partial C_{ct}$ with probability
            \begin{align*}
                P\{\mathcal{N}^{\alpha}(t)=0\}=1/E_{\alpha, 1}(\lambda t^{\alpha}),
            \end{align*}
            uniformly distributed
            on the circumference $\partial C_{ct}$ because of the initial uniformly distributed orientations of motion.
        \end{os}

        \begin{os}
            \label{binoz}
            A generalization of the planar random motion treated in \citet{kol} and
            related to the fractional version treated above
            can be described as follows.
            A homogeneous Poisson process governs the changes of
            direction occurring at times $t_j$, with $0<t_1<\dots t_j<\dots<t_n< t$, of a
            particle moving with velocity $c$. At time $t_j$ the
            particle takes the orientation $\theta_j$ uniformly
            distributed in $[0, 2\pi]$. The position $(\mathcal{X}(t), \mathcal{Y}(t))$
            of the randomly moving particle, after $B(n,\alpha)$
            changes of direction, reads
            \begin{equation}
                \begin{cases}
                    \mathcal{X}(t)=\displaystyle\sum\limits_{j=0}^{B(n,\alpha)}c\left(t_j-t_{j-1}\right)\cos\theta_j,\\ \\
                    \mathcal{Y}(t)=\displaystyle\sum\limits_{j=0}^{B(n,\alpha)}c\left(t_j-t_{j-1}\right)\sin\theta_j,
                \end{cases}
            \end{equation}
            where $B(n,\alpha)$, $0<\alpha\leq 1$ is a binomial r.v.
            independent from $\theta_j$ and $t_j$, $0\leq j\leq
            B(n,\alpha)$. The conditional distribution of $(\mathcal{X}(t), \mathcal{Y}(t))$
            reads
            \begin{equation}
                \label{bino}
                P\{\mathcal{X}(t)\in dx, \mathcal{Y}(t)\in dy|\mathcal{N}(t)=B(n,\alpha)\}
                =\frac{B(n,\alpha)}{2\pi(ct)^{B(n,\alpha)}}\left(c^2t^2-x^2-y^2\right)^{\frac{B(n,\alpha)}{2}-1},
            \end{equation}
            where $\mathcal{N}(t)$, $t\geq 0$ denotes the number of
            events in $[0,t]$ of a homogeneous Poisson process.
            This corresponds to randomize formula (11) of
            \citep{kol} with $B(n,\alpha)$. In this case the
            parameter $\alpha$ is the order of the operator
            appearing in \eqref{2d}.\\
            The mean value of \eqref{bino} becomes
            \begin{align}
                \label{bin1}
                \mathbb{E}P\{\mathcal{X}(t)\in dx, \mathcal{Y}(t)\in dy|\mathcal{N}(t)=B(n,\alpha)\}&= \sum_{k=0}^n
                \frac{k\left(c^2t^2-x^2-y^2\right)^{\frac{k}{2}-1}}{2\pi(ct)^k}
                \binom{n}{k}\alpha^k(1-\alpha)^{n-k}\\
                \nonumber &= \sum_{k=0}^n \frac{k\left[\sqrt{1-\frac{x^2+y^2}{c^2t^2}}\right]^k}{2\pi}
                \binom{n}{k}\frac{\alpha^k(1-\alpha)^{n-k}}{c^2t^2-x^2-y^2}\\
                \nonumber &= \sum_{k=1}^n\frac{n!}{2\pi(n-k)!(k-1)!}\left(\alpha
                \sqrt{1-\frac{x^2+y^2}{c^2t^2}}\right)^k \frac{(1-\alpha)^{n-k}}
                {c^2t^2-x^2-y^2}\\
                \nonumber &= \frac{n}{2\pi} \sum_{k=0}^{n-1}\binom{n-1}{k}\left(\alpha
                \sqrt{1-\frac{x^2+y^2}{c^2t^2}}\right)^{k+1}
                \frac{(1-\alpha)^{n-1-k}}{c^2t^2-x^2-y^2}\\
                \nonumber &=\frac{n\alpha}{2\pi\sqrt{c^2t^2-x^2-y^2}}\frac{1}{(ct)^n}
                \left[ct+\alpha\left(\sqrt{c^2t^2-x^2-y^2}-ct\right)\right]^{n-1},
            \end{align}
            for $n\geq 1$.
            Note that, for $\alpha =1$ we retrieve distribution \eqref{pt1}.

            The unconditional distribution related to \eqref{bin1}, obtained
            by randomizing $n$ with a fractional Poisson process
            $\mathcal{N}^{\alpha}(t)$,
             becomes
            \begin{align}
                P\{\mathcal{X}(t)\in dx, \mathcal{Y}(t)\in dy\}&=\sum_{n=0}^{\infty}
                \mathbb{E}P\{\mathcal{X}(t)\in dx, \mathcal{Y}(t)
                \in dy|\mathcal{N}(t)=B(n,\alpha)\}P\{\mathcal{N}^{\alpha}(t)=n\}\\
                \nonumber &=\frac{\lambda}{2\pi c}\frac{E_{\alpha, \alpha}
                \left(\frac{\lambda}{c}(ct+\alpha\{\sqrt{c^2t^2-x^2-y^2}-ct\})\right)}
                {\sqrt{c^2t^2-x^2-y^2}E_{\alpha,1}(\lambda t)}.
            \end{align}
            In the case where the process governing the number $n$ in the binomial r.v.\ of
            changes of direction is an homogeneous Poisson process, we
            have instead
            \begin{align}
                P\{X(t)\in dx, Y(t)\in dy\}=\frac{\lambda\alpha}{2\pi c}\frac{e^{-\frac{\lambda\alpha}{c}
                \left[ct-\sqrt{c^2t^2-x^2-y^2}
                \right]}}{\sqrt{c^2t^2-x^2-y^2}}.
            \end{align}

            Note that the following interesting inequality holds:
            \begin{align*}
                \mathbb{E} P \{ \mathcal{X}(t) \in dx, \mathcal{Y}(t) \in dy
                |\mathcal{N}(t)= B(n,\alpha) \} \ge \alpha^n P \{ X(t) \in dx, Y(t) \in dy
                | \mathcal{N}(t)=n \}.
            \end{align*}

            We finally observe that the conditional distribution \eqref{conpl} of
            the fractional planar random motion
            \begin{equation}
                P\{X^{\alpha}(t)\in dx, Y^{\alpha}(t)\in dy|\mathcal{N}^{\alpha}(t)=n\}
                =\frac{\alpha n}{2\pi(ct)^{\alpha n}}\left(c^2t^2-x^2-y^2
                \right)^{\frac{n\alpha}{2}-1}, \quad (x,y)\in C_{ct},
            \end{equation}
            can be obtained as
            \begin{equation}
                P\{\mathcal{X}(t)\in dx, \mathcal{Y}(t)\in dy|\mathcal{N}(t)=\mathbb{E}(B(n,\alpha))=n\alpha\}.
            \end{equation}
            Thus the fractionality implies that we take a fraction
            $\alpha$ of the number of changes of direction of the
            classical planar random motion.
        \end{os}

        \begin{lem}
            \label{sealand}
            The function
            \begin{align*}
                u_\alpha(w)=\sum_{k=0}^{\infty}
                \left(\frac{\lambda}{c^{\alpha}}\right)^{2k+1}
                \frac{w^{2k\alpha+\alpha-2}}{\Gamma(2k\alpha+\alpha)}=\frac{\lambda w^{\alpha-2}}{c^{\alpha}}
                E_{2\alpha,\alpha}\left(\frac{\lambda^2 w^{2\alpha}}{c^{2\alpha}}\right),
            \end{align*}
           is an analytic solution of the equation
           \begin{equation}
                \left(\frac{d^2}{dw^2}+\frac{2}{w}\frac{d}{dw}\right)^{\alpha}u_\alpha(w)
                =\frac{\lambda^2}{c^{2\alpha}}u_\alpha(w)+
                \frac{\lambda}{c^{\alpha}}\frac{w^{-2-\alpha}}{\Gamma(-\alpha)}.
            \end{equation}
        \end{lem}
        \begin{proof}
            In view of \eqref{cont}, we have that
            \begin{align*}
                \left(\frac{d^2}{dw^2}+\frac{2}{w}\frac{d}{dw}\right)^{\alpha}w^{2k\alpha+\alpha-2}
                = \frac{\Gamma(2k\alpha+\alpha)}{\Gamma(2k\alpha-\alpha)}
                w^{2k\alpha-\alpha-2}.
            \end{align*}
            Hence
            \begin{align}
                \nonumber\left(\frac{d^2}{dw^2}+\frac{2}{w}\frac{d}{dw}\right)^{\alpha}u_\alpha(w)&=\sum_{k=0}^{\infty}
                \left(\frac{\lambda}{c^{\alpha}}\right)^{2k+1}
                \frac{w^{2k\alpha-\alpha-2}}{\Gamma(2k\alpha-\alpha)}\\
                \nonumber &=\left(\frac{\lambda}{c^{\alpha}}\right)^{2}\sum_{k'=0}^{\infty}
                \left(\frac{\lambda}{c^{\alpha}}\right)^{2k'+1}
                \frac{(w)^{2k'\alpha+\alpha-2}}{\Gamma(2k'\alpha+\alpha)}
                + \frac{\lambda}{c^{\alpha}}\frac{w^{-\alpha-2}}{\Gamma(-\alpha)}\\
                \nonumber &=\left(\frac{\lambda}{c^{\alpha}}\right)^2u_\alpha(w)
                + \frac{\lambda}{c^{\alpha}}\frac{w^{-\alpha-2}}{\Gamma(-\alpha)},
            \end{align}
            as claimed.
        \end{proof}

        \begin{coro}
            The function
            \begin{align}
                \label{cor1}
                u_\alpha(x,y,t)&=\frac{2\pi c^{\alpha}}{\lambda}E_{\alpha, 1}(\lambda t^{\alpha})\sum_{k=0}^{\infty}
                P\{X^{\alpha}(t)\in dx, Y^{\alpha}(t)\in dy|
                \mathcal{N}^{\alpha}(t)= k\}P\{\mathcal{N}^{\alpha}(t)=k\}\\
                \nonumber &=  \sum_{k=0}^{\infty}
                \left(\frac{\lambda}{c^{\alpha}}\right)^{k}
                \frac{(\sqrt{c^2t^2-x^2-y^2})^{k\alpha-2}}{\Gamma(k\alpha)},
            \end{align}
            is an analytic solution to the non-homogeneous fractional Klein--Gordon equation
            \begin{equation}
                \left(\frac{\partial^2}{\partial t^2}-c^2 \left\{\frac{\partial^2}{\partial x^2}
                +\frac{\partial^2}{\partial y^2}\right\}\right)^{\alpha}u_\alpha(x,y,t)
                =\lambda^2 u_\alpha(x,y,t)+
                \frac{\lambda c^{\alpha}\left(\sqrt{c^2t^2-x^2-y^2} \right)^{-\alpha-2}}{\Gamma(-\alpha)}.
            \end{equation}
        \end{coro}
        \begin{proof}
            By splitting function \eqref{cor1} into even-order and odd-order
            terms and by considering Theorem \ref{berger} and Lemma \ref{sealand}, the result of corollary
            above immediately follows.
        \end{proof}

        We now study the distribution of the projection on the $x$-axis of the vector processes above, namely
        $X^{\alpha}_p(t)$, $t\geq 0$.
        We note that the distribution of $X^{\alpha}_p(t)$ possesses only the absolutely continuous component,
        unlike its two-dimensional counterpart.
        We first obtain
        \begin{align}
            \label{marg}
            \int_{-\sqrt{c^2t^2-x^2}}^{+\sqrt{c^2t^2-x^2}}\frac{k\alpha\left(\sqrt{c^2t^2-(x^2+y^2)}\right)^{\alpha k-2}}
            {2\pi(ct)^{\alpha k}}dy&=\frac{k\alpha\left(\sqrt{c^2t^2-x^2}\right)^{k\alpha-1}}{2\pi(ct)^{\alpha k}}
            \int_0^1\left(1-z\right)^{\frac{\alpha
            k}{2}-1}z^{-\frac{1}{2}}dz\\
            \nonumber &=\frac{k\alpha\left(\sqrt{c^2t^2-x^2}\right)^{k\alpha-1}}{2\pi(ct)^{\alpha
            k}}\frac{\Gamma(\frac{k\alpha}{2})\Gamma(\frac{1}{2})}{\Gamma(\frac{k\alpha+1}{2})}\\
            \nonumber &=\frac{\left(\sqrt{c^2t^2-x^2}\right)^{k\alpha-1}}{(2 ct)^{\alpha
            k}}\frac{\Gamma(k\alpha+1)}{[\Gamma(\frac{k\alpha+1}{2})]^2},
        \end{align}
        where we used the fact that
        \begin{align}
            \nonumber &\frac{k\alpha}{2}\Gamma\left(\frac{k\alpha}{2}\right)=\Gamma\left(\frac{k\alpha}{2}
            +\frac{1}{2}+\frac{1}{2}\right)
            =\frac{\Gamma(\frac{1}{2})2^{1-2(\frac{k\alpha+1}{2})}\Gamma(\alpha k+1)}{\Gamma(\frac{\alpha
            k+1}{2})}.
        \end{align}
        In view of \eqref{marg} we extract the density
        \begin{align}
            p^{\alpha}(x,t)&=\frac{1}{\pi\sqrt{c^2t^2-x^2}}\frac{1}{E_{\alpha,1}(\lambda
            t^{\alpha})}+\frac{1}{E_{\alpha,1}(\lambda
            t^{\alpha})}\frac{\lambda}{2^\alpha c^{\alpha}}\sum_{k=0}^{\infty}
            \left(\frac{\lambda}{2^\alpha c^{\alpha}}\right)^{k}
            \frac{\left(\sqrt{c^2t^2-x^2}\right)^{k\alpha+\alpha-1}}{[\Gamma(\frac{\alpha
            k+\alpha+1}{2})]^2}\\
            \nonumber &=\frac{1}{E_{\alpha,1}(\lambda
            t^{\alpha})}\sum_{k=-1}^{\infty}\left(\frac{\lambda}{2^\alpha c^{\alpha}}\right)^{k+1}
            \frac{\left(\sqrt{c^2t^2-x^2}\right)^{k\alpha+\alpha-1}}{[\Gamma(\frac{\alpha
            k+\alpha+1}{2})]^2}\\
            \nonumber&=\frac{1}{E_{\alpha,1}(\lambda
            t^{\alpha})}\sum_{k'=0}^{\infty}\left(\frac{\lambda}{2^\alpha c^{\alpha}}\right)^{k'}
            \frac{\left(\sqrt{c^2t^2-x^2}\right)^{k'\alpha-1}}{[\Gamma(\frac{\alpha k'+1}{2})]^2}\\
            \nonumber &= \frac{1}{\sqrt{c^2t^2-x^2}}\frac{E_{2;\frac{\alpha}{2},
            \frac{1}{2}}\left(\frac{\lambda}{2^\alpha c^{\alpha}}
            (\sqrt{c^2t^2-x^2})^{\alpha}\right)}{E_{\alpha,1}(\lambda t^{\alpha})}.
        \end{align}
        We observe that the components of the planar fractional
        telegraph-type process have distributions without singular part
        because the singular part is projected on the $x$-axis.
        Moreover we notice that we recover for $\alpha =1$, the known
        case discussed, for example, by \citet{ale2}.

        Let us consider the function
        \begin{align}
            \label{new}
            u_\alpha(x,t)= E_{\alpha,1}(\lambda
            t^{\alpha})p^{\alpha}(x,t)=\sum_{k=0}^{\infty}\left(\frac{\lambda}{2^\alpha c^{\alpha}}\right)^{k}
            \frac{(\sqrt{c^2t^2-x^2})^{k\alpha-1}}{[\Gamma(\frac{\alpha
            k+1}{2})]^2}.
        \end{align}
        For $u_\alpha(x,t)$, expressed in terms of $w=\sqrt{c^2t^2-x^2}$, we have the following theorem.
        \begin{te}
            The function \eqref{new} is an analytic solution of the equation
            \begin{equation}
                \label{pal}
                \left(\frac{d^2}{dw^2}+\frac{1}{w}\frac{d}{dw}\right)^{\alpha}u_\alpha(w)=
                \left(\frac{\lambda}{c^{\alpha}}\right)^2
                u_\alpha(w)+\frac{w^{-1-2\alpha}}{[\Gamma(\frac{1-2\alpha}{2})]^2}+\frac{\lambda}{c^{\alpha}}
                \frac{w^{-1-\alpha}}{[\Gamma(\frac{1-\alpha}{2})]^2}.
            \end{equation}
        \end{te}
        \begin{proof}
            By exploiting the relationship \eqref{pri} for $\beta = \alpha k-1$, we arrive at \eqref{new}.
        \end{proof}

        \begin{os}
            We observe that in the case $\alpha =1$, the function
            \begin{align*}
                u_1(w) = \sum_{k=0}^{\infty}\left(\frac{\lambda}{2c}\right)^{k}
                \frac{w^{k-1}}{[\Gamma(\frac{
                k+1}{2})]^2},
            \end{align*}
            is an analytic solution of the inhomogeneous Bessel-type
            equation
            \begin{equation}
                \left(\frac{d^2}{dw^2}+\frac{1}{w}\frac{d}{dw}\right)u_1(w)=
                \frac{\lambda^2}{c^2}u_1(w)+\frac{1}{w^3[\Gamma(-\frac{1}{2})]^2}.
            \end{equation}
            On the other hand, for $\alpha=1/2$, the first non-homogeneous term in \eqref{pal}
            vanishes and we have that the function
            \begin{equation}
                u_{1/2}(w) = \sum_{k=0}^{\infty}\left(\frac{\lambda}{2c}\right)^{k}
                \frac{w^{\frac{k}{2}-1}}{[\Gamma(\frac{k}{4}+\frac{1}{2})]^2}
            \end{equation}
            is an analytic solution of the inhomogeneous Bessel-type
            equation
            \begin{equation}
                \left(\frac{d^2}{dw^2}+\frac{1}{w}\frac{d}{dw}\right)^{1/2}u_{1/2}(w)=
                \frac{\lambda^2}{c}u_{1/2}(w)+\frac{\lambda}{\sqrt{c}}\frac{1}{\sqrt{w^3}[\Gamma(\frac{1}{4})]^2}.
            \end{equation}
        \end{os}

    \section{$N$-Dimensional fractional random flights}

        We now treat the general $N$-dimensional fractional Klein--Gordon equation, i.e.\
        \begin{equation}
            \label{ddim}
            \left(\frac{\partial^2}{\partial t^2}-c^2\Delta \right)^{\alpha}u_\alpha (\mathbf{x},t)
            =\lambda^2 u_\alpha(\mathbf{x},t), \qquad \alpha \in (0,1], \: \mathbf{x}\in \mathbb{R}^N.
        \end{equation}
        By means of the transformation
        \begin{align*}
            w =\left(c^2t^2-\sum_{k=1}^N x_k^2 \right)^{1/2},
        \end{align*}
        where $x_k$ is the $k$-th coordinate of the $N$-dimensional vector $\mathbf{x}$, we transform \eqref{ddim} in
        \begin{equation}
            \label{Lddi}
            \left(\frac{d^2}{dw^2}+\frac{N}{w}\frac{d}{dw}\right)^{\alpha}u_\alpha(w)
            =\frac{\lambda^2}{c^{2\alpha}}u_\alpha(w).
        \end{equation}
        The operator appearing in \eqref{Lddi} can be considered again as a
        specific case of the operator \eqref{L} with $a_1=-N$, $a_2=N$,
        $a_3=0$, $a=0$, $n=m=2$, $b_1=\frac{N-1}{2}$ and $b_2=0$. Hence, from \eqref{pot} we have
        that
        \begin{equation}
            \left(\frac{d^2}{dw^2}+\frac{N}{w}\frac{d}{dw}\right)^{\alpha}u_\alpha (w)=4^{\alpha}w^{-2\alpha}
            I_2^{0,-\alpha}I_2^{\frac{N-1}{2}, -\alpha}u_\alpha (w).
        \end{equation}
        We are now ready to state the following
        \begin{te}
            \label{parafarmacia}
            A solution to \eqref{Lddi} is given by
            \begin{equation}
                \label{Ndim}
                u_\alpha (w)= \sum_{k=0}^{\infty}\left(\frac{\lambda}{2^\alpha c^{\alpha}}\right)^{2k}
                \frac{w^{2\alpha k+2\alpha-2}}{\Gamma(\alpha k+\alpha+\frac{N-1}{2})\Gamma(\alpha k+\alpha)}.
            \end{equation}
        \end{te}
        \begin{proof}
            We first observe that
            \begin{align}
                \nonumber &\left(\frac{d^2}{dw^2}+\frac{N}{w}\frac{d}{dw}\right)^{\alpha}w^{\beta}=4^{\alpha}w^{-2\alpha}
                I_2^{0,-\alpha}I_2^{\frac{N-1}{2}, -\alpha}w^{\beta}\\
                \nonumber &=4^{\alpha}w^{-2\alpha}
                I_2^{0,-\alpha}\left(\frac{N-1}{2}+1-\alpha+\frac{\beta}{2}\right)
                \frac{\Gamma(\frac{N-1}{2}+1+\frac{\beta}{2})}
                {\Gamma(1-\alpha+\frac{N-1}{2}+\frac{\beta}{2}+1)}w^{\beta}\\
                \nonumber &=4^{\alpha}w^{-2\alpha}
                \left(1-\alpha+\frac{\beta}{2}\right)\left(\frac{N-1}{2}+1-\alpha+\frac{\beta}{2}\right)
                \frac{\Gamma(\frac{\beta}{2}+1)}{\Gamma(1-\alpha+\frac{\beta}{2}+1)}
                \frac{\Gamma(\frac{N-1}{2}+1+\frac{\beta}{2})}{\Gamma(1-\alpha+\frac{N-1}{2}
                +\frac{\beta}{2}+1)}w^{\beta}\\
                \nonumber &= 4^{\alpha}w^{\beta-2\alpha}\frac{\Gamma(\frac{\beta}{2}+1)}
                {\Gamma(1-\alpha+\frac{\beta}{2})}
                \frac{\Gamma(\frac{N-1}{2}+1+\frac{\beta}{2})}
                {\Gamma(1-\alpha+\frac{N-1}{2}+\frac{\beta}{2})}.
            \end{align}
            Hence we have that
            \begin{align}
                \nonumber &\left(\frac{d^2}{dw^2}+\frac{N}{w}\frac{d}{dw}\right)^{\alpha}\sum_{k=0}^{\infty}
                \left(\frac{\lambda}{2^\alpha c^{\alpha}}\right)^{2k}
                \frac{w^{2\alpha k+2\alpha-2}}{\Gamma(\alpha k+\alpha+\frac{N-1}{2})\Gamma(\alpha k+\alpha)}\\
                \nonumber &=4^{\alpha}\sum_{k=0}^{\infty}
                \left(\frac{\lambda}{2^\alpha c^{\alpha}}\right)^{2k}
                \frac{w^{2\alpha k-2}}{\Gamma(\alpha k+\frac{N-1}{2})\Gamma(\alpha k)}\\
                \nonumber &=\left(\frac{\lambda}{c^{\alpha}}\right)^{2}\sum_{k'=0}^{\infty}
                \left(\frac{\lambda}{2^\alpha c^{\alpha}}\right)^{2k'}
                \frac{w^{2\alpha k'+2\alpha-2}}{\Gamma(\alpha k'+\alpha+\frac{N-1}{2})\Gamma(\alpha k'+\alpha)},
            \end{align}
            as claimed.
        \end{proof}

        \begin{os}
            From Theorem \ref{parafarmacia}, writing the function \eqref{Ndim} in
            terms of the variables $(\mathbf{x},t)$, we have that
            \begin{equation}
                u_\alpha(\mathbf{x},t)=\sum_{k=0}^{\infty}\left(\frac{\lambda}{2^\alpha c^{\alpha}}\right)^{2k}
                \frac{\left(\sqrt{c^2t^2-\sum_{k=1}^N x_k^2}\right)^{2\alpha k+2\alpha-2}}
                {\Gamma(\alpha k+\alpha+\frac{N-1}{2})\Gamma(\alpha k+\alpha)}
            \end{equation}
            solves the $N$-dimensional fractional Klein--Gordon equation
            \begin{equation}
                \left(\frac{\partial^2}{\partial t^2}-c^2\Delta\right)^{\alpha}
                u_\alpha(\mathbf{x},t)=\lambda^2 u_\alpha(\mathbf{x},t).
            \end{equation}
            Moreover, for $N=1$ and $N=2$, we recover the results found
            above, up to a multiplicative term.
        \end{os}

        \begin{os}
            Inspired by Theorem \ref{parafarmacia}, we construct a fractional
            extension of the random flight in $\mathbb{R}^4$, previously
            studied in \citet{ale2}. For the vector process
            $\mathbf{X}(t)=(X_1(t),X_2(t),X_3(t),X_4(t))$ we write
            \begin{equation}
                \label{4co}
                P\Biggl\{\bigcap_{j=1}^4 \{ X_j(t)\in
                dx_j \} \Bigr|\mathcal{N}^{\alpha/2}(t)=k\Biggr\}
                =\frac{\left(\frac{k\alpha}{2}+1\right)\frac{k\alpha}{2}\left(
                \sqrt{c^2t^2-\|\mathbf{x}\|^2}\right)^{k\alpha-2}}{\pi^2(ct)^{k\alpha+2}}\prod_{j=1}^{4}dx_j,
            \end{equation}
            for $\alpha\in (1,2]$ and $\|\mathbf{x}\|^2\leq c^2t^2$. We note
            that, for $\alpha =2$, \eqref{4co} coincides with formula (1.5)
            of \citet{ale2}. The distribution of the fractional Poisson
            process $\mathcal{N}^{\alpha/2}(t)$, $t\geq 0$, reads
            \begin{equation}
                \label{4do}
                P\{\mathcal{N}^{\alpha/2}(t)=k\}=\frac{\left(\lambda
                t^{\alpha/2}\right)^k}{\Gamma\left(\frac{k\alpha}{2}+1\right)}
                \frac{1}{E_{\alpha/2,1}(\lambda t^{\alpha/2})}, \qquad k\geq 0, \:
                \alpha \in(1,2].
            \end{equation}
            The absolutely continuous component of the distribution of
            $\mathbf{X}(t)$ is given by
            \begin{align}
                \label{4dp}
                &p^{\alpha}(\mathbf{x},t)=\sum_{k=0}^{\infty}\frac{\left(\frac{k\alpha}{2}+1\right)\frac{k\alpha}{2}\left(
                \sqrt{c^2t^2-\|\mathbf{x}\|^2}\right)^{k\alpha-2}}{\pi^2(ct)^{k\alpha+2}}
                \frac{\left(\lambda t^{\alpha/2}\right)^k}
                {\Gamma\left(\frac{k\alpha}{2}+1\right)}
                \frac{1}{E_{\alpha/2,1}(\lambda t^{\alpha/2})}\\
                \nonumber &=\frac{1}{\pi^2 E_{\alpha/2,1}(\lambda
                t^{\frac{\alpha}{2}})(ct)^2}\left[\sum_{k=0}^{\infty}
                \left(\frac{\lambda}{c^{\alpha}t^{\alpha/2}}\right)^k\frac{\left(
                \sqrt{c^2t^2-\|\mathbf{x}\|^2}\right)^{k\alpha-2}}{\Gamma\left(\frac{k\alpha}{2}-1\right)}
                +2\sum_{k=0}^{\infty}\left(\frac{\lambda}{c^{\alpha}t^{\alpha/2}}\right)^k\frac{\left(
                \sqrt{c^2t^2-\|\mathbf{x}\|^2}\right)^{k\alpha-2}}{\Gamma\left(\frac{k\alpha}{2}\right)}\right]\\
                \nonumber &=\frac{\lambda}{\pi^2 c^{2+\alpha}t^{2+\alpha/2}E_{\alpha/2,1}(\lambda
                t^{\alpha/2})\left(
                \sqrt{c^2t^2-\|\mathbf{x}\|^2}\right)^{2-\alpha}}\bigg[E_{\alpha/2,\alpha/2-1}
                \left(\frac{\lambda}{c^{\alpha}t^{\alpha/2}}\left(
                \sqrt{c^2t^2-\|\mathbf{x}\|^2}\right)^{\alpha}\right)\\
                \nonumber & \quad +2E_{\alpha/2,\alpha/2}\left(\frac{\lambda}{c^{\alpha}t^{\alpha/2}}\left(
                \sqrt{c^2t^2-\|\mathbf{x}\|^2}\right)^{\alpha}\right)\bigg].
            \end{align}
            For $\alpha=2$, formula \eqref{4dp} coincides with (3.7) of
            \citet{ale2}, since
            \begin{align*}
                E_{1,0}(x)=xe^x,\qquad E_{1,1}(x)=e^x.
            \end{align*}
            The fractional random flight in $\mathbb{R}^4$ can be viewed as
            a motion with a binomial number $B(n,\frac{\alpha}{2})$ of
            changes of orientations (uniformly distributed on the
            hypersphere), where $n$ possesses fractional Poisson
            distribution given by \eqref{4do}. The conditional distribution
            \eqref{4co} is thus written as (see Remark \ref{binoz} for the planar case)
            \begin{align*}
                P\Biggl\{\bigcap_{j=1}^4 \{ X_j(t)\in
                dx_j \} \Bigr|\mathcal{N}(t)=\mathbb{E}B\left(n,\frac{\alpha}{2}\right)\Biggr\}.
            \end{align*}
        \end{os}

    \section{Higher order cases}

        We devote this section to the fractional hyper-Bessel operators
        \begin{equation}
             (L_{B_n})^{\alpha}= \Biggl(\frac{1}{w^{n}}\Biggl(\underbrace{w\frac{d}{dw}w
             \frac{d}{dw}\dots w\frac{d}{dw}}_{\text{$n$
             times}}\Biggr)\Biggr)^{\alpha}, \qquad \alpha \in (0,1].
        \end{equation}
        We first treat in detail the fractional third-order Bessel
        equation:
        \begin{equation}
            (L_{B_3})^{\alpha}f(w)=\left(\frac{1}{w^2}\frac{d}{dw}+\frac{3}{w}\frac{d^2}{dw^2}
            +\frac{d^3}{dw^3}\right)^{\alpha}f(w)
            =\left(\frac{\lambda}{c^{\alpha}}\right)^3 f(w).\qquad \alpha\in(0,1].
        \end{equation}
        The operator $L_{B_3}$ coincides with \eqref{L}, for $n=3$, $a_1=-2$, $a_2=a_3=1$, $a_4=0$.
        Therefore, in view of Lemma \ref{duepuntouno}, we obtain that
        $a=0$, $m=3$, $b_1=b_2=b_3=0$. Indeed
        \begin{align*}
            L_{B_3}=\frac{1}{w^2}\frac{d}{dw}\left(w\frac{d}{dw}\right)^2.
        \end{align*}
        Therefore
        \begin{align}
            \left(\frac{1}{w^2}\frac{d}{dw}+\frac{3}{w}\frac{d^2}{dw^2}+\frac{d^3}{dw^3}\right)^{\alpha}f(w)&=
            3^{3\alpha}w^{-3\alpha}\prod_{k=1}^{3}I_3^{b_k,-\alpha}f(w)\\
            \nonumber &=3^{3\alpha}w^{-3\alpha}I_3^{0,-\alpha}I_3^{0,-\alpha}I_3^{0,-\alpha}f(w).
        \end{align}

        \begin{te}
            Let $\alpha\in(0,1]$, then the equation
            \begin{align*}
                \left(\frac{1}{w^2}\frac{d}{dw}+\frac{3}{w}\frac{d^2}{dw^2}+\frac{d^3}{dw^3}\right)^{\alpha}f(w)
                = f(w), \quad w\in \mathbb{R},
            \end{align*}
            is satisfied by
            \begin{equation}
                f(w)=w^{3\alpha-3}\sum_{k=0}^{\infty}\left(\frac{w}{3}\right)^{3\alpha k}
                \frac{1}{[\Gamma(\alpha k+\alpha)]^3}.
            \end{equation}
        \end{te}
        \begin{proof}
            We first observe that for $\beta >0$
            \begin{align}
                (L_{B_3})^{\alpha}w^{\beta}&=3^{3\alpha}w^{-3\alpha}I_3^{0,-\alpha}I_3^{0,-\alpha}
                I_3^{0,-\alpha}w^{\beta}\\
                \nonumber& =3^{3\alpha}w^{-3\alpha}\left[(1-\alpha)I_3^{0, 1-\alpha}+\frac{1}{3}
                I_3^{0, 1-\alpha}\left(w\frac{d}{dx}\right)\right]^3 w^{\beta}\\
                \nonumber& =3^{3\alpha}w^{-3\alpha}(1-\alpha+\frac{1}{3}\beta)^3
                I_3^{0,1-\alpha}I_3^{0,1-\alpha}I_3^{0,1-\alpha}w^{\beta}\\
                \nonumber &=3^{3\alpha}w^{-3\alpha}(1-\alpha+\frac{1}{3}\beta)^3
                \left[\frac{\Gamma\left(\frac{\beta}{3}+1\right)}
                {\Gamma\left(1-\alpha+1+\frac{\beta}{3}\right)}\right]^3 w^{\beta}\\
                \nonumber &=3^{3\alpha}\left[\frac{\Gamma\left(\frac{\beta}{3}+1\right)}
                {\Gamma\left(1-\alpha+\frac{\beta}{3}\right)}\right]^3 w^{\beta-3\alpha}.
            \end{align}
            Then we immediately have
            \begin{align}
                (L_{B_3})^{\alpha}\left(w^{3\alpha-3}\sum_{k=0}^{\infty}\frac{1}{3^{3\alpha k}}
                \frac{w^{3\alpha k}}{[\Gamma(\alpha k+\alpha)]^3}\right)=
                \nonumber 3^{3\alpha}\sum_{k=0}^{\infty}\frac{1}{3^{3\alpha k}}
                \frac{w^{3\alpha k-3}}{[\Gamma(\alpha k)]^3}=
                \nonumber \sum_{k'=0}^{\infty}\frac{1}{3^{3\alpha k}}
                \frac{w^{3\alpha k'+3\alpha-3}}{[\Gamma(\alpha k'+\alpha)]^3}.
            \end{align}
        \end{proof}

        We finally arrive at the general theorem that can be easily proved as the preceding one.
        \begin{te}
            Let be $\alpha\in(0,1]$, $n\in \mathbb{N}$. The equation
            \begin{equation}
                \label{hoe}
                \left(\frac{1}{w^{n}}\left(w\frac{d}{dw}\right)^{n}\right)^{\alpha}f(w)= f(w),
            \end{equation}
            is satisfied by
            \begin{equation}
                f(w)=w^{n\alpha-n}\sum_{k=0}^{\infty}\left(\frac{w}{n}\right)^{n\alpha k}
                \frac{1}{[\Gamma(\alpha k+\alpha)]^n}.
            \end{equation}
        \end{te}

        \begin{os}
            We observe that the fractional higher order equation
            \begin{equation}
                \label{best}
                \left(\frac{1}{w^2}\frac{d}{dw}+\frac{3}{w}\frac{d^2}{dw^2}+\frac{d^3}{dw^3}\right)^{\alpha}f(w)=
                \left(\frac{\sqrt[3]{6}\lambda}{c}\right)^{3\alpha}f(w)
            \end{equation}
            is directly related to the fractional partial differential equation in $2+1$ variables
            \begin{equation}
                \label{o3}
                \left(\frac{\partial}{\partial t}+c\frac{\partial^{\alpha}}{\partial x}\right)^{\alpha}
                \left(\frac{\partial}{\partial t}-\frac{c}{2}\frac{\partial}{\partial x}
                +\frac{c\sqrt{3}}{2}\frac{\partial}{\partial y}\right)^{\alpha}
                \left(\frac{\partial}{\partial t}-\frac{c}{2}\frac{\partial}{\partial x}
                -\frac{c\sqrt{3}}{2}\frac{\partial}{\partial y}\right)^{\alpha}u_\alpha(x,y,t)
                = \lambda^3 u_\alpha(x,y,t).
            \end{equation}
            Indeed it can be obtained from \eqref{o3} by a sequence  of two transformations (see \cite{Or}).
            The first one is
            \begin{align}\nonumber
                \begin{cases}
                    z_1=\frac{ct}{2}+x,\\
                    z_2=\frac{ct-x}{\sqrt{3}}+y,\\
                    z_3=\frac{ct-x}{\sqrt{3}}-y,
                \end{cases}
            \end{align}
            that reduces \eqref{o3} to
            \begin{equation}
                \left(\frac{9c^3}{2}\frac{\partial}{\partial z_1}\frac{\partial}{\partial z_2}
                \frac{\partial}{\partial z_3 }\right)^{\alpha}u_\alpha(z_1,z_2,z_3)=\lambda^3 u_\alpha(z_1,z_2,z_3).
            \end{equation}
            Then, by means of a second transformation $w=\sqrt[3]{z_1 z_2 z_3}$, we obtain \eqref{best}.
            In turn, by considering
            \begin{align*}
                w'=\frac{\sqrt[3]{6}\lambda w}{c},
            \end{align*}
            the equation \eqref{best} can be converted in \eqref{hoe}.
            Therefore the solution of \eqref{o3} can be written as
            \begin{equation}
                \label{solho}
                u_\alpha (x,y,t)=
                \sum_{k=0}^{\infty}\left(\frac{\lambda}{3c}\sqrt[3]{(ct+3x)[(ct-x)^2-3y^2]}\right)^{3\alpha
                k+3\alpha-3}\frac{1}{[\Gamma(\alpha k+\alpha]^3}.
            \end{equation}
            A specific case of \eqref{solho} for $\alpha=1$ reads
            \begin{equation}
                I_{0,3}\left(\frac{\lambda}{c}\sqrt[3]{(ct+3x)[(ct-x)^2-3y^2]}\right),
            \end{equation}
            where
            \begin{align*}
                I_{0,3}(x)=\sum_{k=0}^{\infty}\left(\frac{x}{3}\right)^{3k}\frac{1}{(k!)^3},
            \end{align*}
            is the third-order Bessel function.

            Equation \eqref{o3} emerges in the study of a planar cyclic
            random motion with three directions \citep{Or}. The
            fractional version of this random motion can be obtained from
            its integer counterpart by introducing a randomization of the
            number of changes of direction as done in the previously
            analyzed cases.
        \end{os}

    \section{An application to the fractional Euler--Poisson--Darboux equation}

        We now introduce the following fractional formulation of the
        classical Euler--Poisson--Darboux equation
        \citep[see for example][]{e-p}
        \begin{equation}
            \label{dar}
            \left(\frac{1}{t^{\chi}}\frac{\partial}{\partial t}t^{\chi}\frac{\partial}{\partial t}\right)^{\alpha}
            f(\mathbf{x},t)=
            \left(\frac{\partial^2}{\partial t^2}+\frac{\chi}{t}\frac{\partial}{\partial t}\right)^{\alpha}
            f(\mathbf{x},t)=
            \Delta f(\mathbf{x},t),
        \end{equation}
        where $\alpha\in(0,1]$, $t\geq 0$, $\mathbf{x}\in\mathbb{R}^N$ and $\chi \in \mathbb{R}$.
        Clearly the operator appearing in \eqref{dar} is a special
        case of \eqref{L}, with $n=2$, $a_1= -\chi$, $a_2= \chi$ and
        $a_3=0$.
        In the
        following we take for simplicity $\chi=1$.
        This is a fractional generalization of the classical equation that is recovered for $\alpha=1$.
        Using the formalism used in the previous section, we can write \eqref{dar} in a compact way as
        \begin{equation}
            (L_B)_t^{\alpha}f(\mathbf{x},t)=
            \Delta f(\mathbf{x},t).
        \end{equation}
        Applying the Fourier transform we have
        \begin{equation}
            (L_B)_t^{\alpha}f(\mathbf{k},t)=
            |\mathbf{k}|^2 f(\mathbf{k},t),
        \end{equation}
        whose solution is given by
        \begin{equation}
            f(\mathbf{k},t)= t^{2\alpha-2}\sum_{j=0}^{\infty}\left[\left(\frac{t}{2}\right)^{\alpha}
            |\mathbf{k}|\right]^{2j}\frac{1}{[\Gamma(\alpha j+\alpha)]^2}.
        \end{equation}
        In more general, but rather formal way, we have the following
        \begin{te}
            Consider the initial value problem (IVP)
            \begin{equation}
                \label{lag}
                \begin{cases}
                    \left(\frac{\partial^2}{\partial t^2}+\frac{1}{t}\frac{\partial}{\partial t}\right)^{\alpha}
                    f(x,t)= \widehat{O}_x f(x,t), & t>0,\\
                    f(x,0)=g(x),
                \end{cases}
            \end{equation}
            where $\widehat{O}_x$ is an integro-differential operator acting on the space variable that satisfies
            the semigroup property and
            $g(x)$ is an analytic function. Then the operational solution of equation \eqref{lag} is given by:
            \begin{equation}
                \label{op}
                f(x,t)=t^{2\alpha-2}\sum_{j=0}^{\infty}\left(\frac{t}{2}\right)^{2\alpha j}
                \frac{\widehat{O}_x^j\,g(x)}{[\Gamma(\alpha j+\alpha)]^2}
            \end{equation}
        \end{te}
        The operational solution \eqref{op} becomes an effective solution
        when the series converges, and this depends upon the actual form of
        the initial condition $g(x)$. Operational methods to solve Euler--Poisson--Darboux equations
        are applied in \cite{hai}.

        \smallskip

        \textbf{Acknowledgement.} We are very greatful to both referees
        for their suggestions and, in particular, to reviewer $\sharp
        1$
        for insightful remarks and for checking many calculations.

\end{document}